\documentclass[final,leqno,letterpaper]{etna}

\usepackage[T1]{fontenc}
\usepackage{graphicx}
\usepackage{lineno,hyperref}
\usepackage{epsfig}
\usepackage{geometry}
\usepackage{color}
\usepackage{amscd}
\usepackage{amsmath, amsfonts, amssymb, mathrsfs}
\usepackage{multirow}
\usepackage{enumerate}
\usepackage{verbatim}
\usepackage{tikz}
\usetikzlibrary{matrix,arrows}
\usepackage[justification=raggedright]{caption}
\usepackage[ruled]{algorithm2e}
\usepackage[export]{adjustbox}
\usepackage{wasysym}
\usepackage{subcaption}
\usepackage{booktabs}

\def \vec#1{{\bf{#1}}}

\newcommand{\bi}{\begin{itemize}}
\newcommand{\ei}{\end{itemize}}

\newcommand{\diverg}{\vec{\nabla}\cdot}
\newcommand{\director}{\vec{n}}
\newcommand{\curl}{\vec{\nabla}\times}
\newcommand{\Ltwoinner}[3]{\langle #1,#2 \rangle_0}
\newcommand{\Ltwonorm}[2]{\Vert #1 \Vert_0}

\newcommand{\Ltwoinnerndim}[4]{\langle #1,#2 \rangle_0}

\newcommand{\Rone}{\mathbb{R}}

\newcommand{\diff}[1]{\, d#1}

\newcommand{\Hone}[1]{H^1(#1)}

\newcommand{\Honenot}[1]{H^1_0({#1})}

\newcommand{\Ltwo}[1]{L^2(#1)}

\newcommand{\Lp}[1]{L^p (\Omega)}

\newcommand{\Linfinity}[1]{L^{\infty}(\Omega)}

\newcommand{\triangulation}{\mathcal{T}_h}

\newcommand{\diam}{\text{diam }}

\graphicspath{{Figures/}}

\setbibdata{00}{00}{00}{2025} 

\hypersetup{%
    pdftitle={Document title},
    pdfauthor={John Doe, Erwin Schr\"odinger},
    pdfkeywords={a keyword, another keyword, and another one}
    }

\title{An A Posteriori Error Estimator for Electrically Coupled Liquid Crystal Equilibrium Configurations
}

\author{J.H. Adler\footnotemark[2]
        \and D.B. Emerson\footnotemark[3]}

\shorttitle{An Error Estimator for Electrically Coupled LCs} 
\shortauthor{J.H. Adler and D.B. Emerson}

\begin{document}

\maketitle

\renewcommand{\thefootnote}{\fnsymbol{footnote}}

\footnotetext[2]{Tufts University, Medford MA 02155, USA ({\tt james.adler@tufts.edu})}
\footnotetext[3]{Vector Institute, Toronto ON M5G 1M1, Canada ({\tt david.emerson@vectorinstitute.ai})}

\begin{abstract}
This paper derives an a posteriori error estimator for the nonlinear first-order optimality conditions associated with the electrically and flexoelectrically coupled Frank-Oseen model of liquid crystals, building on previous results for elastic systems. The estimator is proposed for a penalty approach to imposing the unit-length constraint required by the model. Moreover, theory is proven establishing that the estimator provides a reliable estimate of global approximation error and an efficient measure of local error, suitable for use in adaptive refinement. Numerical experiments demonstrate significant improvements in efficiency with adaptive refinement guided by the proposed estimator in a multilevel, nested-iteration framework and superior physical properties for challenging electrically coupled systems.
\end{abstract}



\section{Introduction}

As materials possessing mesophases with characteristics of both liquids and organized solids, liquid crystals exhibit many interesting physical properties inspiring extensive study and a wide range of applications. In addition to considerable use in modern display technologies, liquid crystals are used for nanoparticle organization \cite{Lagerwall1},  manufacture of structured nanoporous solids \cite{Wan1}, and efficient conversion of mechanical strain to electrical energy \cite{Harden1}, among many others.

The focus of this paper is nematic liquid crystals, which are rod-like molecules with long-range orientational order described by a vector field $\director(x, y, z) = (n_1, n_2, n_3)^T$, referred to as the director. For the model considered here, $\director$ is constrained to unit-length pointwise throughout the domain, $\Omega$. In addition to their elastic properties, liquid crystals are dielectrically active such that their structures are affected by the presence of electric fields. In addition, certain types of liquid crystals demonstrate flexoelectric coupling wherein deformations of the director produce internally generated electric fields \cite{Meyer1}. Thorough treatments of liquid crystal physics are found in \cite{Stewart1, Virga1}.

With the combination of highly-coupled physics and complicated experimental behavior, numerical simulations of liquid crystal structures are fundamental to the study of novel physical phenomena, optimization of device design, and analysis of experimental observations \cite{Emerson5, Clerc1, RojasGomez1}. As many applications and experiments require simulations on two- and three-dimensional domains with complicated boundary conditions, the development of highly efficient and accurate numerical approaches is important. Effective a posteriori error estimators significantly increase the efficiency of numerical methods for partial differential equations (PDEs) and variational systems by guiding the construction of optimal discretizations via adaptive refinement. A wealth of research exists for the design and theoretical support of effective error estimators in the context of finite-element methods. This includes techniques treating both linear and nonlinear PDEs across a number of applications (see, e.g., \cite{John4, Oden1, Verfurth3, Bank1, Babuska2}).

In \cite{Emerson6}, a reliable a posterior error estimator was developed for the first-order optimality conditions arising from minimization of the Frank-Oseen elastic free-energy model. Using the estimator to guide adaptive mesh refinement (AMR) in numerical simulations, produced competitive solutions in terms of constraint conformance and free energy with considerably less computational work. In this paper, the elastic error estimator is extended to consider systems with electric and flexoelectric coupling. The proposed, coupled, a posteriori error estimator is shown to be a reliable estimate of global approximation error and an efficient indicator of local error. The incorporation of electric and flexoelectric fields produce a markedly more complicated estimator and auxiliary spaces requiring careful theoretical treatment. In certain cases, modifications of the original theory in \cite{Emerson6} are straightforward. In others, however, important adjustments must be made to ensure that the framework of \cite{Verfurth1, Verfurth2} remains viable. Numerical experiments leveraging a multilevel nested-iteration framework for problems with both external and flexoelectrically induced electric fields demonstrate the performance of the estimator compared with uniform refinement.

This paper is organized as follows. In Section \ref{model}, the coupled Frank-Oseen free-energy model and associated variational system for the first-order optimality conditions is introduced. Additional notation and prerequisite theoretical results to be applied in the reliability and efficiency proofs are discussed in Section \ref{preliminaries}. In Section \ref{theory}, the specific coupled error estimator is derived and proofs of reliability and efficiency are constructed. Section \ref{numerics} presents numerical experiments demonstrating the performance of the error estimator and a multilevel AMR approach. Finally, Section \ref{conclusions} provides some concluding remarks and a discussion of future work.

\section{Free-Energy Model and Optimality Conditions} \label{model}

Liquid crystals are simulated using a number of different models \cite{Davis1, Onsager1, Gartland1}. Here, the Frank-Oseen free-energy model is considered where, for a domain $\Omega$, the coupled equilibrium free energy is composed of three main components associated with elastic deformations, external electric fields, and flexoelectrically generated fields. 
The coupled free-energy functional is then written
\begin{align}
\mathcal{G}(\director, \phi) &= \frac{1}{2}K_1 \Ltwonorm{\diverg \director}{\Omega}^2 + \frac{1}{2}K_3\Ltwoinnerndim{\vec{Z} \curl \director}{\curl \director}{\Omega}{3}  - \frac{1}{2}\epsilon_0\epsilon_{\perp}\Ltwoinnerndim{\nabla \phi}{\nabla \phi}{\Omega}{3} \label{flexofunctional} \\
&\hspace{0.35cm} - \frac{1}{2}\epsilon_0 \epsilon_a \Ltwoinner{\director \cdot \nabla \phi}{\director \cdot \nabla \phi}{\Omega} + e_s \Ltwoinner{\diverg \director}{\director \cdot \nabla \phi}{\Omega} + e_b\Ltwoinnerndim{\director \times \curl \director}{\nabla \phi}{\Omega}{3}, \nonumber
\end{align}
where we denote the classical $\Ltwo{\Omega}$ inner product and norm as $\Ltwoinner{\cdot}{\cdot}{\Omega}$ and $\Vert \cdot \Vert_0$, respectively, for both scalar and vector quantities. The variable $\phi$ in \eqref{flexofunctional} denotes the electric potential for an electric field, $\vec{E}$, such that $\vec{E} = \nabla \phi$. Here, $K_i \geq 0$, $i=1,2, 3$ are the Frank constants, which depend on the physical characteristics of the liquid crystal and have a significant impact on orientational structure \cite{Atherton2, Lee1}. Assuming that each $K_i \neq 0$, we define the tensor $\vec{Z} = \vec{I} - (1-\kappa) \director \otimes \director$, where $\kappa = K_2/K_3$. The permittivity of free space is denoted by $\epsilon_0 > 0$. The dielectric anisotropy of the liquid crystal is $\epsilon_a = \epsilon_{\parallel} - \epsilon_{\perp}$, with the constants $\epsilon_{\parallel}, \epsilon_{\perp} > 0$ representing the parallel and perpendicular dielectric permittivity, respectively. For positive $\epsilon_a$, the director favors parallel alignment with the electric field, while negative anisotropy indicates a perpendicular preference. Finally, $e_s$ and $e_b$ are material constants specifying the liquid crystal's flexoelectric response. 

Equilibrium states correspond to configurations that minimize the functional in \eqref{flexofunctional} subject to the local unit-length constraint, $\director \cdot \director -1 = 0$, on $\Omega$. Additionally, the relevant Maxwell's equations for a static electric field, $\diverg \vec{D} = 0$ and $\curl \vec{E} = \vec{0}$, known as Gauss' and Faraday's laws, respectively, must be satisfied. For this system, 
\begin{align*}
\vec{D} = -\epsilon_0 \epsilon_{\perp} \nabla \phi - \epsilon_0 \epsilon_a (\director \cdot \nabla \phi) \director + e_s \director (\diverg \director) + e_b (\director \times \curl \director).
\end{align*}
Note that the use of an electric potential implies that Faraday's law is trivially satisfied, and it is straightforward to show that a minimizing pair, $(\director_*, \phi_*)$, adhering to the unit-length constraint, satisfies Gauss' law in weak form.
For a full derivation of the functional in \eqref{flexofunctional}, see \cite{Emerson2, Emerson5}. Throughout this paper, the presence of Dirichlet boundary conditions is assumed. Therefore the functional has been simplified using the null Lagrangian discussed in \cite{Stewart1}. Moreover, the free-energy expression has been non-dimensionalized using the approach detailed in \cite{Emerson3}.

Throughout this paper, it is assumed that $\director \in \left (H_{\vec{g}_1}^1(\Omega)\right )^3 = \{ \vec{v} \in \left (H^1(\Omega)\right )^3 : \vec{v} = \vec{g}_1 \text{ on } \partial \Omega \}$ and $\phi \in H_{g_2}^1(\Omega) = \{ \psi \in H^1(\Omega) : \psi = g_2 \text{ on } \partial \Omega \}$, where $H^1(\Omega)$ denotes the standard Sobolev space with norm $\Vert \cdot \Vert_1$. The boundary functions $\vec{g}_1$ and $g_2$ are assumed to satisfy appropriate compatibility conditions for the domain. Note that if $\vec{g}_1 = \vec{0}$, the space $\left (H^1_{\vec{g}_1}(\Omega)\right )^3 = \left (H^1_0(\Omega)\right )^3$.

The penalty approach studied in \cite{Emerson2, Emerson3} is used to enforce the pointwise unit-length constraint. The penalty method adds a weighted, positive term to the free-energy functional, penalizing deviation from the constraint such that for $\zeta > 0$
\begin{align*}
\mathcal{H}(\director, \phi) = \mathcal{G}(\director, \phi) + \frac{1}{2}\zeta \Ltwoinner{\director \cdot \director -1}{\director \cdot \director -1}{\Omega}.
\end{align*}
Taking the first variation of $\mathcal{H}(\director, \phi)$, the first-order optimality conditions are written
\begin{align}
\mathcal{P}(\director, \phi) = \mathcal{C}(\director, \phi) + 2\zeta \Ltwoinner{\vec{v} \cdot \director}{\director \cdot \director -1}{\Omega}  = 0 && \forall (\vec{v}, \psi) \in H^1_0(\Omega)^3 \times H^1_0(\Omega), \label{PenaltyFOOC}
\end{align}
where
\begin{align*}
&\mathcal{C}(\director, \phi) = K_1\Ltwoinner{\diverg \director}{\diverg \vec{v}}{\Omega} + K_3\Ltwoinnerndim{\vec{Z} \curl \director}{\curl \vec{v}}{\Omega}{3} - \epsilon_0 \epsilon_a \Ltwoinner{\director \cdot \nabla \phi}{\vec{v} \cdot \nabla \phi}{\Omega} \nonumber \\
& \qquad \qquad \qquad+ (K_2-K_3)\Ltwoinner{\director \cdot \curl \director}{\vec{v} \cdot \curl \director}{\Omega}  - \epsilon_0 \epsilon_{\perp}\Ltwoinnerndim{\nabla \phi}{\nabla \psi}{\Omega}{3} - \epsilon_0 \epsilon_a \Ltwoinner{\director \cdot \nabla \phi}{\director \cdot \nabla \psi}{\Omega} \nonumber \\
& \qquad \qquad \qquad + e_s\big( \Ltwoinner{\diverg \director}{\vec{v} \cdot \nabla \phi}{\Omega} + \Ltwoinner{\diverg \vec{v}}{\director \cdot \nabla \phi}{\Omega} \big)  + e_b\big( \Ltwoinnerndim{\director \times \curl \vec{v}}{\nabla \phi}{\Omega}{3} \nonumber \\
&\qquad \qquad \qquad+ \Ltwoinnerndim{\vec{v} \times \curl \director}{\nabla \phi}{\Omega}{3} \big) + e_s \Ltwoinner{\diverg \director}{\director \cdot \nabla \psi}{\Omega} +e_b  \Ltwoinnerndim{\director \times \curl \director}{\nabla \psi}{\Omega}{3}.
\end{align*}
In \cite{Emerson6}, a posteriori error estimators were proposed for the first-order optimality conditions of purely elastic liquid crystal systems. Below, the elastic estimator is extended to include electric and flexoelectric coupling. Moreover, this estimator is shown to be both reliable and locally efficient.

\section{Preliminary Theory and Notation} \label{preliminaries}

In this section, some additional notation and requisite existing theoretical results used in subsequent sections are discussed. For the theory to follow, it is assumed that the domain $\Omega$ is open and connected, with a polyhedral boundary. For any open subset $\omega \subset \Omega$ with Lipschitz boundary, norms restricted to the subdomain are denoted with an index as $\Vert \cdot \Vert_{1, \omega}$ and $\Vert \cdot \Vert_{0, \omega}$. Let $\{\triangulation\}$, $0 < h \leq 1,$ be a shape-regular family of meshes subdividing $\Omega$. That is, there exists a $\rho > 0$ such that, for all $T \in \triangulation$ and $h \in (0, 1]$,
\begin{align}
\max \{\diam T : T \in \triangulation\} \leq h \, \diam \Omega, & & \diam B_T \geq \rho  \, \diam T \label{shaperegular},
\end{align}
where $B_T$ is the largest ball contained in $T$ such that $T$ is star-shaped with respect to $B_T$. In addition, we assume that any family of meshes satisfies the admissibility condition such that any two cells of $\triangulation$ are either disjoint or share a complete, smooth sub-manifold of their boundaries. For any $T \in \triangulation$, let $h_T = \diam T$, denote the set of edges of $T$ as $\mathcal{E}(T)$, and let $h_E = \diam E$ for $E \in \mathcal{E}(T)$. It is also assumed that the mesh family is fine enough that $h_T, h_E \leq 1$ for each mesh. Note that the shape-regularity conditions of \eqref{shaperegular} ensures that the ratio $h_T/h_E$ is bounded above and below by constants independent of $h$, $T$, and $E$ and implies that the smallest angle of any $T$ is bounded from below by a constant independent of $h$ \cite{Verfurth2}.

The sets of vertices corresponding to $T$ and $E$ are written $\mathcal{N}(T)$ and $\mathcal{N}(E)$, respectively. The set of all edges for $\triangulation$ is written $\mathcal{E}_h = \bigcup_{T \in \triangulation} \mathcal{E}(T)$, and $\mathcal{E}_{h, \Omega}$ signifies the subset of interior edges. Finally, some specific subdomains of $\Omega$ are written
\begin{align*}
\omega_T &= \bigcup_{\mathcal{E}(T) \cap \mathcal{E}(T') \neq \emptyset} T', & \omega_E &= \bigcup_{E \in \mathcal{E}(T')} T',\\
\tilde{\omega}_T &= \bigcup_{\mathcal{N}(T) \cap \mathcal{N}(T') \neq \emptyset} T', & \tilde{\omega}_E &= \bigcup_{\mathcal{N}(E) \cap \mathcal{N}(T') \neq \emptyset} T'.
\end{align*}
For the meshes, define a fixed reference element $\hat{T}$ and reference edge $\hat{E}$ as $\hat{T} = \{\hat{x} \in \mathbb{R}^n : \sum_{i = 1}^n \hat{x}_i \leq 1, \hat{x}_j \geq 0, 1 \leq j \leq n \}$ and $\hat{E} = \hat{T} \cap \{ \hat{x} \in \mathbb{R}^n : \hat{x}_n = 0 \}$. The mesh is assumed to be affine equivalent such that, for any $T \in \triangulation$, there exists an invertible affine mapping from the reference components to $T$. For any $E \in \mathcal{E}_h$, we assign a unit normal vector $\eta_E$. By convention, $\eta_E$ coincides with the outward normal for any $E$ on the domain boundary, $\partial \Omega$. Then, for any piecewise continuous function $\psi$, the jump across $E$ in the direction $\eta_E$ is denoted as $[\psi]_E$. Finally, for $k \in \mathbb{N}$, define the finite-dimensional space 
\begin{align*}
S_h^{k, 0} &= \{\psi: \Omega \rightarrow \Rone : \psi \vert_T \in \Pi_k,  \forall T \in \triangulation \} \cap C(\bar{\Omega}),
\end{align*}
where $\Pi_k$ is the set of polynomials of degree at most $k$, $\psi \vert_T$ is the restriction of $\psi$ to the element $T$, and $C(\bar{\Omega})$ is the collection of continuous functions on the closure of $\Omega$.

Making use of the notation and assumptions established above, a collection of important supporting theoretical results is gathered in this section and referenced in the efficiency and reliability theory developed in Section \ref{theory}. Let $I_h: L^1(\Omega) \rightarrow S_h^{1, 0}$ denote the Cl\'{e}ment interpolation operator \cite{Clement1, Verfurth1}. Then, the following approximation error bound holds for $\triangulation$.

\begin{lemma} \label{clementlemma}
For any $T \in \triangulation$ and $E \in \mathcal{E}_h$
\begin{align*}
\Vert \psi - I_h \psi \Vert_{0, T} &\leq C_1 h_T \Vert \psi \Vert_{1, \tilde{\omega}_T} & \forall \psi \in \Hone{\tilde{\omega}_T}, \\
\Vert \psi - I_h \psi \Vert_{0, E} &\leq C_2 h_E^{1/2} \Vert \psi \Vert_{1, \tilde{\omega}_E} & \forall \psi \in \Hone{\tilde{\omega}_E},
\end{align*}
where $C_1$ and $C_2$ depend only on the shape-regularity condition in \eqref{shaperegular}. 
\end{lemma}

Following the notation in \cite{Verfurth1, Verfurth2}, let $\Psi_{\hat{T}}, \Psi_{\hat{E}} \in C^{\infty}(\hat{T}, \Rone)$ be cut-off functions defined on the reference components $\hat{T}$ and $\hat{E}$ such that
\begin{align*}
&0 \leq \Psi_{\hat{T}} \leq 1, \quad \max_{\hat{x} \in \hat{T}} \Psi_{\hat{T}}(\hat{x}) = 1, \quad \Psi_{\hat{T}} = 0 \text{ on } \partial \hat{T}, \\
&0 \leq \Psi_{\hat{E}} \leq 1, \quad \max_{\hat{x} \in \hat{E}} \Psi_{\hat{E}}(\hat{x}) = 1, \quad \Psi_{\hat{E}} = 0 \text{ on } \partial \hat{T} \backslash \hat{E}.
\end{align*}
Define a continuation operator $\hat{P}: L^{\infty}(\hat{E}) \rightarrow L^{\infty}(\hat{T})$ as 
\begin{align*}
\hat{P} \hat{u}(\hat{x}_1, \ldots, \hat{x}_n) := \hat{u}(\hat{x}_1, \ldots, \hat{x}_{n-1}),
\end{align*}
for all $\hat{x} \in \hat{T}$, and fix two arbitrary finite-dimensional subspaces, $V_{\hat{T}} \subset L^{\infty}(\hat{T})$ and $V_{\hat{E}} \subset L^{\infty}(\hat{E})$. Applying the affine mappings from reference components, corresponding functions, $\Psi_T$ and $\Psi_E$, operator $P: L^{\infty}(E) \rightarrow L^{\infty}(T)$, and spaces $V_T$ and $V_E$ are defined for arbitrary $T \in \triangulation$ and $E \in \mathcal{E}_h$ with analogous properties. Thus, the following lemma and corollary hold, c.f. \cite{Verfurth1, Verfurth2, Brenner1}.
\begin{lemma} \label{cutoffinequalities}
There are constants $C_1, \ldots, C_7$ depending only on the finite-dimensional spaces $V_{\hat{T}}$ and $V_{\hat{E}}$, the functions $\Psi_{\hat{T}}$ and $\Psi_{\hat{E}}$, and  the shape-regularity bounds of \eqref{shaperegular} such that for all $T \in \triangulation$, $E \in \mathcal{E}(T)$, $u \in V_T$, and $\sigma \in V_E$
\begin{align}
C_1 \Vert u \Vert_{0, T} &\leq \sup_{v \in V_T} \frac{\int_T u \Psi_T v \diff{V}}{\Vert v \Vert_{0, T}} \leq \Vert u \Vert_{0, T}, \label{cutoff1}\\
C_2 \Vert \sigma \Vert_{0, E} &\leq \sup_{\tau \in V_E} \frac{\int_E \sigma \Psi_E \tau \diff{S}}{\Vert \tau \Vert_{0, E}} \leq \Vert \sigma \Vert_{0, E}, \label{cutoff2}\\
C_3 h_T^{-1} \Vert \Psi_T u \Vert_{0, T} &\leq \Vert \nabla (\Psi_T u) \Vert_{0, T}  \leq C_4 h_T^{-1} \Vert \Psi_T u \Vert_{0, T}, \nonumber \\
C_5 h_T^{-1} \Vert \Psi_E P\sigma \Vert_{0, T} &\leq \Vert \nabla (\Psi_E P \sigma) \Vert_{0, T}  \leq C_6 h_T^{-1} \Vert \Psi_E P \sigma \Vert_{0, T}, \nonumber \\
\Vert \Psi_E P \sigma \Vert_{0, T} &\leq C_7 h_T^{1/2} \Vert \sigma \Vert_{0, E}. \label{cutoff5}
\end{align}
\end{lemma}
Note that with shape-regularity of the mesh, after proper adjustment of $C_i$ in any of the above inequalities, the mesh constant $h_T$ may be exchanged for $h_E$ while maintaining the inequality.
\begin{corollary} \label{cutoffexpansion}
Under the assumptions of Lemma \ref{cutoffinequalities}, there exist $\bar{C}_4, \bar{C}_6 > 0$, independent of $h$, such that
\begin{align}
\Vert \Psi_T u  \Vert_{1,T} &\leq \bar{C}_4 h_T^{-1} \Vert \Psi_T u \Vert_{0, T}, \label{expansion1} \\
\Vert \Psi_E P \sigma \Vert_{1, T} &\leq \bar{C}_6 h_T^{-1} \Vert \Psi_E P \sigma \Vert_{0, T}. \label{expansion2}
\end{align}
\end{corollary}

Finally, we state two key propositions from the framework developed by Verf\"{u}rth \cite{Verfurth1, Verfurth2}. Let $X$ and $Y$ be Banach spaces with norms $\Vert \cdot \Vert_X$ and $\Vert \cdot \Vert_Y$ and denote the space of continuous linear maps from $X$ to $Y$ as $\mathcal{L}(X, Y)$ with the natural operator norm $\Vert \cdot \Vert_{\mathcal{L}(X, Y)}$. The subset of linear homeomorphisms from $X$ to $Y$ is written $\text{Isom}(X, Y)$. Define $Y^* = \mathcal{L}(Y, \Rone)$ to be the dual space of $Y$, with norm $\Vert \cdot \Vert_{Y^*}$, where the associated duality pairing is written $\langle \cdot, \cdot \rangle$. Let $F \in C^1(X, Y^*)$ be a continuously differentiable function for which a solution $u \in X$ is sought such that $F(u) = 0$. Denoting the derivative of $F$ as $DF$ and a ball of radius $R > 0$ centered at $u \in X$ as $B(u, R) = \{ v \in X : \Vert u - v \Vert_X < R \}$, the first proposition is as follows.
\begin{proposition}[\hspace{-4pt} {\cite[Pg. 47]{Verfurth2}}] \label{nonlinearErrorEstimation}
Let $u_0 \in X$ be a regular solution to $F(u) = 0$ in the sense that $DF(u_0) \in \text{Isom}(X, Y^*)$. Assume that $DF$ is Lipschitz continuous at $u_0$, where there exists an $R_0 > 0$ such that
\begin{align*}
\gamma = \sup_{u \in B(u_0, R_0)} \frac{\Vert DF(u) - DF(u_0) \Vert_{\mathcal{L}(X, Y^*)}}{\Vert u - u_0 \Vert_X} < \infty.
\end{align*}
Set $R = \min \big \{ R_0, \gamma^{-1} \Vert DF(u_0)^{-1} \Vert^{-1}_{\mathcal{L}(Y^*, X)}, 2 \gamma^{-1} \Vert DF(u_0) \Vert_{\mathcal{L}(X, Y^*)} \big \}$.
Then, the error estimate
\begin{align*}
\frac{1}{2} \Vert DF(u_0) \Vert_{\mathcal{L}(X, Y^*)}^{-1} \Vert F(u) \Vert_{Y^*} \leq \Vert u - u_0 \Vert_X \leq 2 \Vert DF(u_0)^{-1} \Vert_{\mathcal{L}(Y^*, X)} \Vert F(u) \Vert_{Y^*},
\end{align*}
holds for all $u \in B(u_0, R)$.
\end{proposition}

Let $X_h \subset X$ and $Y_h \subset Y$ be finite-dimensional subspaces and $F_h \in C(X_h, Y_h^*)$ be an approximation of $F$. Consider the discretized problem of finding $u_h \in X_h$ such that $F_h(u_h) = 0$.
\begin{proposition}[\hspace{-4pt} {\cite[Pg. 52]{Verfurth2}}] \label{auxiliarySpaceInequality}
Let $u_h \in X_h$ be an approximate solution to the discretized problem in the sense that $\Vert F_h(u_h) \Vert_{Y_h^*}$ is approximately zero. Assume that there is a restriction operator $R_h \in \mathcal{L}(Y, Y_h)$, a finite-dimensional space $\tilde{Y}_h \subset Y$, and an approximation $\tilde{F}_h : X_h \rightarrow Y^*$ of $F$ at $u_h$ such that 
\begin{align*}
\Vert (\text{Id}_Y - R_h)^*\tilde{F}_h(u_h) \Vert_{Y^*} \leq C_0 \Vert \tilde{F}_h (u_h) \Vert_{\tilde{Y}_h^*},
\end{align*}
where $\text{Id}_Y$ is the identity operator on $Y$, $^*$ indicates application of $(\text{Id}_Y - R_h)$ to the dual variables, and $C_0 > 0$ is independent of $h$. Then the following estimate holds,
\begin{align*}
\Vert F(u_h) \Vert_{Y^*} &\leq C_0 \Vert \tilde{F}_h (u_h) \Vert_{\tilde{Y}^*_h} + \Vert (\text{Id}_Y - R_h)^*[F(u_h) - \tilde{F}_h(u_h)] \Vert_{Y^*} \\
& \hspace{2em} + \Vert R_h \Vert_{\mathcal{L}(Y, Y_h)} \Vert F(u_h) - F_h(u_h) \Vert_{Y_h^*} + \Vert R_h \Vert_{\mathcal{L}(Y, Y_h)} \Vert F_h(u_h) \Vert_{Y_h^*}.
\end{align*}
\end{proposition}
The first result provides an approximation error bound using the residual, while the second yields a concrete set of terms bounding the residual from above. 

\section{A Reliable and Efficient Coupled Error Estimator} \label{theory}

In this section, an a posteriori error estimator is proposed for the first-order optimality conditions of Section \ref{model}, extending the results of \cite{Emerson6} to include electric and flexoelectric coupling. Furthermore, using the theory outlined in the previous section, the estimator is shown to be a reliable estimate of global approximation error and an efficient indicator of local error, suitable for use in AMR schemes. 

To begin, consider the first-order optimality conditions for the penalty method in \eqref{PenaltyFOOC}. Let $Y = X_0 = \left (\Honenot{\Omega}\right )^3 \times \Honenot{\Omega}$ and $X = \left (H_{\vec{g}_1}^1(\Omega)\right )^3 \times H_{g_2}^1(\Omega)$. Then, $\mathcal{P}(\director, \phi) \in C^1(X, Y^*)$, and the Dirichlet boundary conditions imply that for a fixed $(\director, \phi) \in X$, $D\mathcal{P}(\director, \phi): X_0 \rightarrow Y^*$. In discretizing the variational system, we consider general discrete spaces 
\begin{align*}
[S_h^{1, 0}]^3 \subset V_h \subset [S_h^{s, 0}]^3, & & [S_h^{1, 0}] \subset Q_h \subset [S_h^{t, 0}],
\end{align*}
for $s, t \geq 1$ and the finite-dimensional space $Y_h = \{ (\vec{v}_h, \psi_h) \in V_h \times Q_h : \vec{v}_h = \vec{0} \text{ and } \psi_h = 0 \text{ on } \partial \Omega \}$. For the theory presented here, we assume that the imposed boundary conditions on $(\director, \phi)$ are exactly representable on the coarsest mesh of $\{\triangulation\}$. Observe that this assumption on the boundary conditions admits projection of the boundary functions $\vec{g}_1$ and $g_2$ onto the coarsest mesh. Thus, the analysis to follow concerns estimation of the error arising in solution approximations on the interior of $\Omega$ but not from approximation of the boundary conditions. Hence, set $X_h = (V_h \times Q_h \cap X)$. Note that in the numerical results below, any boundary condition functions are interpolated with mesh refinement.

For $(\vec{v}, \psi) \in Y$ and $\langle \mathcal{P}(\director, \phi), (\vec{v}, \psi) \rangle$ define the discrete approximation
\begin{align*}
\langle \mathcal{P}_h(\director_h, \phi_h), (\vec{v}_h, \psi_h) \rangle = \langle \mathcal{P}(\director_h, \phi_h), (\vec{v}_h, \psi_h) \rangle,
\end{align*}
for $(\director_h, \phi_h) \in X_h, (\vec{v}_h, \psi_h) \in Y_h$. For the remainder of this section, assume that the pair $(\director_h, \phi_h)$ is a solution to the discrete problem
\begin{align}
\mathcal{P}_h(\director_h, \phi_h)=0, && \forall (\vec{v}_h, \psi_h) \in Y_h. \label{discretePenaltyFOOC}
\end{align}
In order to simplify notation, define the vector and scalar quantities
\begin{align*}
\vec{p} &= -K_1 \nabla (\diverg \director_h) + K_3 \curl (\vec{Z}(\director_h) \curl \director_h) + (K_2 - K_3)(\director_h \cdot \curl \director_h) \curl \director_h \\
& \qquad  + 2 \zeta ((\director_h \cdot \director_h - 1) \director_h) - \epsilon_0 \epsilon_a ((\director_h \cdot \nabla \phi_h) \nabla \phi_h) + e_s (\diverg \director_h)\nabla \phi_h \\
& \qquad  - e_s \nabla (\director_h \cdot \nabla \phi_h)  + e_b (\curl \director_h \times \nabla \phi_h) + e_b \curl (\nabla \phi_h \times \director_h), \\
q &= \epsilon_0 \epsilon_{\perp} \Delta \phi_h + \epsilon_0 \epsilon_a \diverg ((\director_h \cdot \nabla \phi_h) \director_h) - e_s \diverg ((\diverg \director_h) \director_h) - e_b \diverg (\director_h \times \curl \director_h), \\
\hat{\vec{p}} & =  [K_1(\diverg \director_h) \eta_E + K_3 (\vec{Z}(\director_h)\curl \director_h) \times \eta_E + e_s (\director_h \cdot \nabla \phi_h) \eta_E + e_b ((\nabla \phi_h \times \director_h) \times \eta_E) ]_E,  \\
\hat{q} &=[ - \epsilon_0 \epsilon_{\perp} (\nabla \phi_h \cdot \eta_E) - \epsilon_0 \epsilon_a (\director_h \cdot \nabla \phi_h)(\director_h \cdot \eta_E) + e_s ((\diverg \director_h) \director_h) \cdot \eta_E  + e_b (\director_h \times \curl \director_h) \cdot \eta_E ]_E,
\end{align*}
where $E \in \mathcal{E}_{h, \Omega}$. Integrating $\langle \mathcal{P}(\director_h, \phi_h), (\vec{v}, \psi) \rangle$ by parts elementwise for each $T \in \triangulation$, using the fact that $\vec{v}_h$ and $\psi_h$ are zero on the boundary, and gathering terms yields
\begin{align}
\langle \mathcal{P}(\director_h, \phi_h), (\vec{v}, \psi) \rangle &= \sum_{T \in \triangulation} \int_T \vec{p} \cdot \vec{v} \diff{V} + \int_T q \cdot \psi \diff{V} + \sum_{E \in \mathcal{E}_{h,  \Omega}}  \int_E \hat{\vec{p}} \cdot \vec{v} \diff{S} + \int_E \hat{q} \cdot \psi \diff{S}. \label{penaltyIntegrationByParts}
\end{align}
This form suggests a local estimator,
\begin{align*}
\Theta_T &= \Bigg \{ h_T^2 \left( \Vert \vec{p} \Vert_{0, T}^2 + \Vert q \Vert_{0, T}^2 \right) + \sum_{E \in \mathcal{E}(T) \cap \mathcal{E}_{h, \Omega}} h_E \left ( \Vert \hat{\vec{p}} \Vert_{0, E}^2 + \Vert \hat{q} \Vert_{0, E}^2 \right) \Bigg \}^{1/2},
\end{align*}
for any $T \in \triangulation$. Note that if no external electric field or flexoelectric coupling is present, $\Theta_T$ collapses to the elastic estimator of \cite{Emerson6}. In addition, the quantity $\Vert q \Vert_{0, T}$ locally measures the solution's conformance to the strong form of Gauss' law.

Let $R_h: Y \rightarrow Y_h$ be a restriction operator such that $R_h(\vec{u}, \varphi) = (I_h u_1, I_h u_2, I_h u_3, I_h \varphi)$ where $I_h$ is the Cl\'{e}ment operator of Lemma \ref{clementlemma}. Further, as no forcing function or Neumann boundary conditions are present, set
\begin{align*}
\langle \tilde{\mathcal{P}}_h(\director_h, \phi_h), (\vec{v}, \psi) \rangle = \langle \mathcal{P}(\director_h, \phi_h), (\vec{v}, \psi) \rangle.
\end{align*}
This definition, along with that of the discrete approximation above Equation \eqref{discretePenaltyFOOC}, implies that 
\begin{align}
\Vert (\text{Id}_Y - R_h)^*[\mathcal{P}(\director_h, \phi_h) - \tilde{\mathcal{P}}_h(\director_h, \phi_h)]_{Y^*} &= 0, \label{zeroComponent1} \\
\Vert \mathcal{P}(\director_h, \phi_h) - \mathcal{P}_h(\director_h, \phi_h) \Vert_{Y_h^*} &= 0. \label{zeroComponent2}
\end{align}
With the above definitions, the following lemma holds.
\begin{lemma} \label{restrictionUpperBound}
There exists a constant $C > 0$, independent of $h$, such that 
\begin{align*}
\Vert (\text{Id}_Y - R_h)^*\tilde{\mathcal{P}}_h(\director_h, \phi_h) \Vert_{Y^*} \leq C \left( \sum_{T \in \triangulation} \Theta_T^2 \right)^{1/2}.
\end{align*}
\end{lemma}
\begin{proof}
First, note that 
\begin{align}
&\Vert (\text{Id}_Y - R_h)^*\tilde{\mathcal{P}}_h(\director_h, \phi_h) \Vert_{Y^*} \nonumber \\
&\qquad = \sup_{\substack{[\vec{v}, \psi] \in Y \\ \Vert [\vec{v}, \psi] \Vert_Y = 1}} \sum_{T \in \triangulation} \sum_{i=1}^3 \int_T p_i \cdot (v_i - I_h v_i) \diff{V} + \int_T q \cdot (\psi - I_h \psi) \diff{V} \nonumber \\ 
& \hspace{8em} + \sum_{E \in \mathcal{E}_{h, \Omega}} \sum_{i = 1}^3 \int_E \hat{p}_i \cdot (v_i - I_h v_i) \diff{S} + \int_E \hat{q} \cdot (\psi - I_h \psi) \diff{S} \nonumber \\
&\qquad \leq \sup_{\substack{[\vec{v}, \psi] \in Y \\ \Vert [\vec{v}, \psi] \Vert_Y = 1}} \sum_{T \in \triangulation} \sum_{i=1}^3 C_1 h_T \Vert p_i \Vert_{0, T} \Vert v_i \Vert_{1, \tilde{\omega}_T} + C_1 h_T \Vert q \Vert_{0, T} \Vert \psi \Vert_{1, \tilde{\omega}_T} \label{CSClementInequality} \\
& \hspace{8em} + \sum_{E \in \mathcal{E}_{h, \Omega}} \sum_{i = 1}^3 C_2 h_E^{1/2} \Vert \hat{p}_i \Vert_{0, E} \Vert v_i \Vert_{1, \tilde{\omega}_E} + C_2 h_E^{1/2} \Vert \hat{q} \Vert_{0, E} \Vert \psi \Vert_{1, \tilde{\omega}_E}, \nonumber
\end{align}
where \eqref{CSClementInequality} is given by applying the Cauchy-Schwarz inequality and Lemma \ref{clementlemma} to the interpolation quantities. Using the Cauchy-Schwarz inequality for sums and letting $\tilde{C} = \max(C_1, C_2)$ implies that
\begin{align*}
& \Vert (\text{Id}_Y - R_h)^*\tilde{\mathcal{P}}_h(\director_h, \phi_h) \Vert_{Y^*} \\
&\qquad\leq \sup_{\substack{[\vec{v}, \psi] \in Y \\ \Vert [\vec{v}, \psi] \Vert_Y = 1}}  \tilde{C} \Bigg ( \sum_{T \in \triangulation} h_T^2 \left( \Vert \vec{p} \Vert_{0, T}^2 + \Vert q \Vert_{0, T}^2 \right) + \sum_{E \in \mathcal{E}_{h, \Omega}} h_E \left( \Vert \hat{\vec{p}} \Vert_{0, E}^2 + \Vert \hat{q} \Vert_{0, E}^2 \right)\Bigg )^{1/2} \\
& \hspace{1.5in} \cdot \Bigg(\sum_{T \in \triangulation} \Vert \vec{v} \Vert_{1, \tilde{\omega}_T}^2 + \Vert \psi \Vert_{1, \tilde{\omega}_T}^2 +  \sum_{E \in \mathcal{E}_{h, \Omega}} \Vert \vec{v} \Vert_{1, \tilde{\omega}_E}^2 + \Vert \psi \Vert_{1, \tilde{\omega}_E}^2\Bigg)^{1/2}.
\end{align*}
Finally, there exists a constant $C_* > 0$ independent of $h$ taking into account repeated elements such that
\begin{align*}
\left (\sum_{T \in \triangulation} \Vert \vec{v} \Vert_{1, \tilde{\omega}_T}^2 + \Vert \psi \Vert_{1, \tilde{\omega}_T}^2 +  \sum_{E \in \mathcal{E}_{h, \Omega}} \Vert \vec{v} \Vert_{1, \tilde{\omega}_E}^2 + \Vert \psi \Vert_{1, \tilde{\omega}_E}^2 \right)^{1/2} \leq C_*  \left \Vert [\vec{v}, \psi] \right \Vert_Y.
\end{align*}
Hence,
\begin{align*}
&\Vert (\text{Id}_Y - R_h)^*\tilde{\mathcal{P}}_h(\director_h, \phi_h) \Vert_{Y^*} \\
&\qquad \leq \sup_{\substack{[\vec{v}, \psi] \in Y \\ \Vert [\vec{v}, \psi] \Vert_Y = 1}} C_* \tilde{C} \Vert [\vec{v}, \psi] \Vert_Y \Bigg( \sum_{T \in \triangulation} h_T^2 \left( \Vert \vec{p} \Vert_{0, T}^2 + \Vert q \Vert_{0, T}^2 \right) + \sum_{E \in \mathcal{E}_{h, \Omega}} h_E \left( \Vert \hat{\vec{p}} \Vert_{0, E}^2 + \Vert \hat{q} \Vert_{0, E}^2 \right) \Bigg)^{1/2} \nonumber \\
&\qquad \leq C \left ( \sum_{T \in \triangulation} \Theta_T^2 \right)^{1/2}.
\end{align*}
The final inequality is obtained by simply noting that the jump components are summed over $E \in \mathcal{E}_{h, \Omega}$.
\end{proof}

Next, define the finite-dimensional auxiliary space $\tilde{Y}_h \subset Y$ as
\begin{align*}
\tilde{Y}_h &= \text{span} \{ [\Psi_T \vec{v}, 0], [\Psi_E P \sigma, 0], [\vec{0}, \Psi_T \psi], [\vec{0}, \Psi_E P \tau] \\
& \hspace{0.8in} : \vec{v} \in [\Pi_{k \vert_T}]^3, \sigma \in [\Pi_{k \vert_E}]^3, \psi \in \Pi_{l \vert_T}, \tau \in \Pi_{l \vert_E}, T \in \triangulation, E \in \mathcal{E}_{h, \Omega} \},
\end{align*}
where $k \geq \max(3s, s+2(t-1))$ and $l \geq 2s + (t-1)$. For this space, the following lemma holds.
\begin{lemma} \label{auxiliarySpaceUpperBound}
There exists a $C > 0$, independent of $h$, such that
\begin{align*}
\Vert \tilde{\mathcal{P}}_h(\director_h, \phi_h) \Vert_{\tilde{Y}_h^*} \leq C \left ( \sum_{T \in \triangulation} \Theta_T^2 \right)^{1/2}.
\end{align*}
\end{lemma}
\begin{proof}
Applying the Cauchy-Schwarz inequality implies that
\begin{align*}
\Vert \tilde{\mathcal{P}}_h(\director_h, \phi_h) \Vert_{\tilde{Y}_h^*} &= \sup_{\substack{[\vec{v_h}, \psi_h] \in \tilde{Y}_h \\ \Vert [\vec{v}_h, \psi_h] \Vert_Y = 1}} \sum_{T \in \triangulation} \int_T \vec{p} \cdot \vec{v}_h \diff{V} + \int_T q \cdot \psi_h \diff{V} \nonumber \\
& \hspace{1.5in} + \sum_{E \in \mathcal{E}_{h, \Omega}} \int_E \hat{\vec{p}} \cdot \vec{v}_h \diff{S} + \int_E \hat{q} \cdot \psi_h \diff{S} \\
& \leq \sup_{\substack{[\vec{v_h}, \psi_h] \in \tilde{Y}_h \\ \Vert [\vec{v}_h, \psi_h] \Vert_Y = 1}} \sum_{T \in \triangulation} \Vert \vec{p} \Vert_{0, T} \Vert \vec{v}_h \Vert_{0, T} + \Vert q \Vert_{0, T} \Vert \psi_h \Vert_{0, T} \\
& \hspace{1.5in} + \sum_{E \in \mathcal{E}_{h, \Omega}} \Vert \hat{\vec{p}} \Vert_{0, E} \Vert \vec{v}_h \Vert_{0, E} + \Vert \hat{q} \Vert_{0, E} \Vert \psi_h \Vert_{0, E}.
\end{align*}
Using the definition of $\tilde{Y}_h$, shape-regularity of the mesh, and standard finite-element scaling arguments,
\begin{align}
\Vert \tilde{\mathcal{P}}_h(\director_h, \phi_h) \Vert_{\tilde{Y}_h^*} & \leq \sup_{\substack{[\vec{v_h}, \psi_h] \in \tilde{Y}_h \\ \Vert [\vec{v}_h, \psi_h] \Vert_Y = 1}} \sum_{T \in \triangulation} C_1 h_T \Vert \vec{p} \Vert_{0, T} \Vert \vec{v}_h \Vert_{1, T} + C_1 h_T \Vert q \Vert_{0, T} \Vert \psi_h \Vert_{1, T} \nonumber \\
& \hspace{1.1in} + \sum_{E \in \mathcal{E}_{h, \Omega}} C_2 h_E^{1/2} \Vert \hat{\vec{p}} \Vert_{0, E} \Vert \vec{v}_h \Vert_{1, \omega_E} + C_2 h_E^{1/2} \Vert \hat{q} \Vert_{0, E} \Vert \psi_h \Vert_{1, \omega_E} \nonumber \\
& \leq \sup_{\substack{[\vec{v_h}, \psi_h] \in \tilde{Y}_h \\ \Vert [\vec{v}_h, \psi_h] \Vert_Y = 1}} \tilde{C} \Bigg( \sum_{T \in \triangulation} h_T^2 \left( \Vert \vec{p} \Vert_{0, T}^2 + \Vert q \Vert_{0, T}^2 \right) + \sum_{E \in \mathcal{E}_{h, \Omega}} h_E \left( \Vert \hat{\vec{p}} \Vert_{0, E}^2 + \Vert \hat{q} \Vert_{0, E}^2 \right) \Bigg)^{1/2} \label{CSSumInequality} \\
& \hspace{1.1in} \cdot \Bigg(  \sum_{T \in \triangulation} \Vert \vec{v}_h \Vert_{1, T}^2 + \Vert \psi_h \Vert_{1, T}^2 + \sum_{E \in \mathcal{E}_{h, \Omega}} \Vert \vec{v}_h \Vert_{1, \omega_E}^2 + \Vert \psi_h \Vert_{1, \omega_E}^2 \Bigg)^{1/2},\nonumber
\end{align}
where $\tilde{C} = \max(C_1, C_2)$ and \eqref{CSSumInequality} is given by the Cauchy-Schwarz inequality for sums. Note, as above, there exists a $C_* > 0$, independent of $h$ and taking into account repeated elements in each sum, such that
\begin{align*}
\Bigg(  \sum_{T \in \triangulation} \Vert \vec{v}_h \Vert_{1, T}^2 + \Vert \psi_h \Vert_{1, T}^2 + \sum_{E \in \mathcal{E}_{h, \Omega}} \Vert \vec{v}_h \Vert_{1, \omega_E}^2 + \Vert \psi_h \Vert_{1, \omega_E}^2 \Bigg)^{1/2} \leq C_* \Vert [\vec{v}_h, \psi_h] \Vert_Y.
\end{align*}
Applying the inequality above to \eqref{CSSumInequality} and using the fact that the supremum is taken over $\Vert [\vec{v}_h, \psi_h] \Vert_Y = 1$,
\begin{align*}
\Vert \tilde{\mathcal{P}}_h(\director_h, \phi_h) \Vert_{\tilde{Y}_h^*} &\leq C_* \tilde{C} \Bigg( \sum_{T \in \triangulation} h_T^2 \left( \Vert \vec{p} \Vert_{0, T}^2 + \Vert q \Vert_{0, T}^2 \right ) + \sum_{E \in \mathcal{E}_{h, \Omega}} h_E \left( \Vert \hat{\vec{p}} \Vert_{0, E}^2 + \Vert \hat{q} \Vert_{0, E}^2 \right) \Bigg)^{1/2} \\
& \leq C \left ( \sum_{T \in \triangulation} \Theta_T^2 \right)^{1/2}.
\end{align*}
As in the previous proof, the last inequality makes use of the fact that the jump components are summed over $E \in \mathcal{E}_{h, \Omega}$.
\end{proof}

The final inequality required to demonstrate reliability of the error estimator is
\begin{align*}
\Vert (\text{Id}_Y - R_h)^*\tilde{\mathcal{P}}_h(\director_h, \phi_h) \Vert_{Y^*} \leq C \Vert \tilde{\mathcal{P}}_h(\director_h, \phi_h) \Vert_{\tilde{Y}_h^*},
\end{align*}
for $C > 0$ and independent of $h$. With the result of Lemma \ref{restrictionUpperBound}, it is sufficient to prove the next lemma.
\begin{lemma} \label{auxilliarySpaceLowerBound}
There exists a $C > 0$, independent of $h$, such that 
\begin{align*}
C \left ( \sum_{T \in \triangulation} \Theta_T^2 \right)^{1/2} \leq \Vert \tilde{\mathcal{P}}_h(\director_h, \phi_h) \Vert_{\tilde{Y}_h^*}.
\end{align*}
\end{lemma}
\begin{proof}
Fix an arbitrary $T \in \triangulation$ and an edge $E \in \mathcal{E}(T) \cap \mathcal{E}_{h, \Omega}$. Further, define a restricted space $\tilde{Y}_{h \vert \omega}$, for $\omega \in \{T, \omega_E, \omega_T \}$, as the set of functions $\vec{f} \in \tilde{Y}_h$ with support such that $\text{supp}(\vec{f}) \subset \omega$. Finally, denote the product spaces $ \left( [\Pi_{k \vert T}]^3 \times \Pi_{l \vert T} \right ) \backslash \{(\vec{0}, 0) \}$ and $\left ( [\Pi_{k \vert E}]^3 \times \Pi_{l \vert E} \right ) \backslash \{(\vec{0}, 0) \}$ as $\Pi_{k, l, T}$, $\Pi_{k, l, E}$, respectively. Note that the constants in this proof correspond to those of Lemma \ref{cutoffinequalities} or Corollary \ref{cutoffexpansion}. First, consider
\begin{align}
C_1 \bar{C}_4 ^{-1} h_T \Vert [\vec{p}, q] \Vert_{0, T} & \leq \sup_{[\vec{w}, u] \in \Pi_{k, l, T}} \bar{C}_4^{-1} h_T \Vert [\Psi_T \vec{w}, \Psi_T u] \Vert_{0, T}^{-1} \int_T (\vec{p}, q) \cdot (\Psi_T \vec{w}, \Psi_T u) \diff{V} \label{C1SupInequality} \\
& \leq \sup_{[\vec{w}, u] \in \Pi_{k, l, T}} \Vert [\Psi_T \vec{w}, \Psi_T u] \Vert_{1, T}^{-1} \int_T (\vec{p}, q) \cdot (\Psi_T \vec{w}, \Psi_T u) \diff{V}. \label{C4PointcareInequality}
\end{align}
The inequality in \eqref{C1SupInequality} is given by applying \eqref{cutoff1} of Lemma \ref{cutoffinequalities}, while the subsequent inequality in \eqref{C4PointcareInequality} relies on \eqref{expansion1} of Corollary \ref{cutoffexpansion}. Noting that both $\Psi_T \vec{w}$ and $\Psi_T u$ vanish at the boundary of $T$,
\begin{align}
C_1 \bar{C}_4 ^{-1} h_T \Vert [\vec{p}, q] \Vert_{0, T} &\leq \sup_{[\vec{w}, u] \in \Pi_{k, l, T}} \Vert [\Psi_T \vec{w}, \Psi_T u] \Vert_{1, T}^{-1} \langle \tilde{\mathcal{P}}_h(\director_h, \phi_h), (\Psi_T \vec{w}, \Psi_T u) \rangle \nonumber \\
& \leq \sup_{\substack{[\vec{v}_h, \psi_h] \in \tilde{Y}_{h \vert T} \\ \Vert [\vec{v}_h, \psi_h] \Vert_Y = 1}} \langle \tilde{\mathcal{P}}_h(\director_h, \phi_h), (\vec{v}_h, \psi_h) \rangle. \label{firstLowerBoundInequality}
\end{align}
Next, by applying \eqref{cutoff2} from Lemma \ref{cutoffinequalities} and observing that the integrals and norms are taken over $E$ where $P$ does not modify the values of either $\sigma$ or $\beta$,
\begin{align*}
C_2 \bar{C}_6^{-1} C_7^{-1} h_{E}^{1/2} \Vert [\hat{\vec{p}}, \hat{q}] \Vert_{0, E} \leq \sup_{[\sigma, \beta] \in \Pi_{k, l, E}} \frac{\bar{C}_6^{-1} h_E}{C_7 h_E^{1/2} \Vert [P \sigma, P \beta] \Vert_{0, E}} \int_E (\hat{\vec{p}}, \hat{q}) \cdot (\Psi_E P \sigma, \Psi_E P \beta)  \diff{S}.
\end{align*}
Now note that $\Psi_E P \sigma$ is supported on $\omega_E$ and that the norm in the denominator is taken over $E$. This implies that
\begin{align}
C_2 \bar{C}_6^{-1} C_7^{-1} h_{E}^{1/2} \Vert [\hat{\vec{p}}, \hat{q}] \Vert_{0, E} &\leq \sup_{[\sigma, \beta] \in \Pi_{k, l, E}} \frac{\bar{C}_6^{-1} h_E}{C_7 h_E^{1/2} \Vert [\sigma, \beta] \Vert_{0, E}} \bigg ( \langle \tilde{\mathcal{P}}_h(\director_h, \phi_h), (\Psi_E P \sigma, \Psi_E P \beta) \rangle \nonumber \\
&\hspace{2.2in} - \int_{\omega_E} (\vec{p}, q) \cdot (\Psi_E P \sigma, \Psi_E P \beta) \diff{V} \bigg)\nonumber \\
& \leq \sup_{[\sigma, \beta] \in \Pi_{k, l, E}} \frac{\bar{C}_6^{-1} h_E}{\Vert [\Psi_E P \sigma, \Psi_E P \beta] \Vert_{0, \omega_E}} \bigg ( \langle \tilde{\mathcal{P}}_h(\director_h, \phi_h), (\Psi_E P \sigma, \Psi_E P \beta) \rangle \label{C7Inequality} \\
&\hspace{2.2in} - \int_{\omega_E} (\vec{p}, q) \cdot (\Psi_E P \sigma, \Psi_E P \beta) \diff{V} \bigg) \nonumber.
\end{align}
The admissibility and shape-regularity properties of the mesh imply that the number of elements in $\omega_E$ is bounded by a constant independent of $h$ and $E$. Thus, \eqref{C7Inequality} is given by \eqref{cutoff5} of Lemma \ref{cutoffinequalities} where $C_7$ has been adjusted to account for the bounded number of elements in $\omega_E$. Distributing the fraction and applying \eqref{expansion2} of Corollary \ref{cutoffexpansion} to the first component and the Cauchy-Schwarz inequality to the second yields
\begin{align}
C_2 \bar{C}_6^{-1} C_7^{-1} h_{E}^{1/2} \Vert [\hat{\vec{p}}, \hat{q}] \Vert_{0, E} &\leq \sup_{[\sigma, \beta] \in \Pi_{k, l, E}} \Vert [\Psi_E P \sigma, \Psi_E P \beta] \Vert_{1, \omega_E}^{-1} \langle \tilde{\mathcal{P}}_h(\director_h, \phi_h), (\Psi_E P \sigma, \Psi_E P \beta) \rangle \nonumber \\
& \hspace{1.5in} + \bar{C}_6^{-1} h_E \sum_{T \in \omega_E} \Vert [\hat{\vec{p}}, \hat{q}] \Vert_{0, T} \nonumber \\
& \leq \sup_{\substack{[\vec{v}_h, \psi_h] \in \tilde{Y}_{h \vert \omega_E} \\ \Vert [\vec{v}_h, \psi_h] \Vert_{Y} =1}} \Vert [\vec{v}_h, \psi_h] \Vert_{1, \omega_E}^{-1} \langle \tilde{\mathcal{P}}_h(\director_h, \phi_h), (\vec{v}_h, \psi_h) \rangle \label{RelateBackToPreviousLowerBound} \\
& \hspace{1.5in} + C_d \sup_{\substack{[\vec{v}_h, \psi_h] \in \tilde{Y}_{h \vert \omega_E} \\ \Vert [\vec{v}_h, \psi_h] \Vert_{Y} =1}} \langle \tilde{\mathcal{P}}_h(\director_h, \phi_h), (\vec{v}_h, \psi_h) \rangle, \nonumber
\end{align}
where the final inequality in \eqref{RelateBackToPreviousLowerBound} is given by expanding the space over which the supremum is taken in the first summand and using the inequality in \eqref{firstLowerBoundInequality}, with $C_d$ relating the constants $\bar{C}_6^{-1} h_E$ and $C_1 \bar{C}_4^{-1} h_T$ and taking care of the summation over $\omega_E$. Specifically, let $C_d = \Gamma \frac{\bar{C}_6 h_E}{C_1 \bar{C}_4^{-1} h_T}$, where $\Gamma$ is the maximum number of elements in $\omega_E$ for any $E$. The admissibility and shape-regularity properties of the mesh ensure that $\Gamma$ is independent of $h$,  $E$, and $T$ and that the ratio $h_E/h_T$ is bounded above and below by independent constants as well. Finally, note that the supremums only increase when taken over $\omega_T$. 

Gathering the bounds in \eqref{firstLowerBoundInequality} and \eqref{RelateBackToPreviousLowerBound} and applying the inequality
\begin{align}
\left ( \sum_i a_i \right )^{1/2} \leq \sum_i a_i^{1/2}, \label{positiveAiLowerBound}
\end{align}
for $a_i > 0$, implies that
\begin{align}
\bar{C} \Theta_T \leq \sup_{\substack{[\vec{v}_h, \psi_h] \in \tilde{Y}_{h \vert \omega_T} \\ \Vert [\vec{v}_h, \psi_h] \Vert_{Y} =1}} \langle \tilde{\mathcal{P}}_h(\director_h, \phi_h), (\vec{v}_h, \psi_h) \rangle. \label{localLowerBound}
\end{align}
Finally, summing over $T \in \triangulation$ and applying \eqref{positiveAiLowerBound} again yields $C \left ( \sum_{T \in \triangulation} \Theta_T^2 \right)^{1/2} \leq \Vert \tilde{\mathcal{P}}_h(\director_h, \phi_h) \Vert_{\tilde{Y}_h^*}.$
\end{proof}

These results enable the statement and proof of the main result of this section establishing reliability and local efficiency of the proposed a posteriori error estimator. 
\begin{theorem}
Say that $(\director_*, \phi_*)$ is a solution to Equation \eqref{PenaltyFOOC} satisfying the assumptions of Proposition \ref{nonlinearErrorEstimation}. Let $(\director_h, \phi_h)$ be a discrete solution to Equation \eqref{discretePenaltyFOOC} such that $\Vert \mathcal{P}_h(\director_h, \phi_h) \Vert_{Y_h^*}  = 0$ and $(\director_h, \phi_h) \in B((\director_*, \phi_*), R)$. Then, there exist $C_r, C_e > 0$, independent of $h$, such that
\begin{align}
&\Vert (\director_*, \phi_*) - (\director_h, \phi_h) \Vert_1 \leq C_r \left ( \sum_{T \in \triangulation} \Theta_T^2 \right)^{1/2}, \label{reliability} \\
&\Theta_T \leq C_e \Vert (\director_*, \phi_*) - (\director_h, \phi_h) \Vert_{1, \omega_T} \label{efficiency}.
\end{align}
\label{final_theorem}
\end{theorem}
\begin{proof}
Combining Lemmas \ref{restrictionUpperBound} and \ref{auxilliarySpaceLowerBound} implies the bound
\begin{align*}
\Vert (\text{Id}_Y - R_h)^* \tilde{\mathcal{P}}_h(\director_h, \phi_h) \Vert_{Y^*} &\leq C_0 \left ( \sum_{T \in \triangulation} \Theta_T^2 \right)^{1/2}  \leq C_1 \Vert \tilde{\mathcal{P}}_h (\director_h, \phi_h) \Vert_{\tilde{Y}_h^*},
\end{align*}
for $C_0, C_1 > 0$. Thus, the conditions of Proposition \ref{auxiliarySpaceInequality} are fulfilled. With the results in Equations \eqref{zeroComponent1} and \eqref{zeroComponent2} and Lemma \ref{auxiliarySpaceUpperBound},
\begin{align*}
\Vert \mathcal{P}(\director_h, \phi_h) \Vert_{Y^*} \leq C_2 \Vert \tilde{\mathcal{P}}_h (\director_h, \phi_h) \Vert_{\tilde{Y}_h^*} \leq C_3 \left ( \sum_{T \in \triangulation} \Theta_T^2 \right)^{1/2}.
\end{align*}
Note that the term $\Vert R_h \Vert_{\mathcal{L}(Y, Y_h)} \Vert F_h(u_h) \Vert_{Y_h^*}$ from Proposition \ref{auxiliarySpaceInequality} is zero when $(\director_h, \phi_h)$ satisfy Equation \eqref{discretePenaltyFOOC}. The upper bound from Proposition \ref{nonlinearErrorEstimation} then implies that
\begin{align*}
\Vert (\director_*, \phi_*) - (\director_h, \phi_h) \Vert_1 &\leq 2 \Vert D \mathcal{P}(\director_*, \phi_*)^{-1} \Vert_{\mathcal{L}(Y^*, X_0)} \Vert \mathcal{P}(\director_h, \phi_h) \Vert_{Y^*} \\
& \leq 2 C_3 \Vert D \mathcal{P}(\director_*, \phi_*)^{-1} \Vert_{\mathcal{L}(Y^*, X_0)} \left ( \sum_{T \in \triangulation} \Theta_T^2 \right)^{1/2}.
\end{align*}
Setting $C_r = 2C_3 \Vert D \mathcal{P}(\director_*, \phi_*)^{-1} \Vert_{\mathcal{L}(Y^*, X_0)}$ proves the inequality in \eqref{reliability}.

As noted in \cite[Remark 2.2]{Verfurth1}, the lower bound of Proposition \ref{nonlinearErrorEstimation} remains valid when restricted to appropriate norms over the open subset $\omega_T \subset \Omega$. Together with inequality \eqref{localLowerBound}, this implies that
\begin{align*}
\Theta_T \leq C_4 \sup_{\substack{[\vec{v}_h, \psi_h] \in \tilde{Y}_{h \vert \omega_T} \\ \Vert [\vec{v}_h, \psi_h] \Vert_{Y} =1}} \langle \tilde{\mathcal{P}}_h(\director_h, \phi_h), (\vec{v}_h, \psi_h) \rangle &\leq C_4 \Vert \mathcal{P}(\director_h, \phi_h) \Vert_{Y^*_{\omega_T}} \\
& \leq 2 C_4 C_5 \Vert (\director_*, \phi_*) - (\director_h, \phi_h) \Vert_{1, \omega_T},
\end{align*} 
where $C_5$ is given by the value of the restriction of the norm $\Vert D \mathcal{P}(\director_*, \phi_*) \Vert_{\mathcal{L}(X_0, Y^*)}$ from Proposition \ref{nonlinearErrorEstimation} to $X_{\omega_T} \subset X_0$ and $Y^*_{\omega_T} \subset Y^*$, the subspaces of $X_0$ and $Y^*$ limited to functions supported on $\omega_T$. Taking $C_e = 2 C_4 C_5$ proves \eqref{efficiency}.
\end{proof}

\begin{remark}
The results of Lemma \ref{cutoffinequalities} are equally applicable to meshes composed of quadrilateral or simplicial elements, as noted in \cite[Remark 3.5]{Verfurth2}. Thus, the results of this section extend to either type of mesh, satisfying equivalent conditions.
\end{remark}

\section{Numerical Results} \label{numerics}

In this section, numerical experiments applying the electrically and flexoelectrically coupled estimator proposed above are presented. The inclusion of both electric and flexoelectric effects, paired with Dirichlet boundary conditions, limits the availability of non-trivial, closed-form solutions. However, the numerical results suggest that the proposed estimator markedly increases simulation efficiency with equivalent or superior performance across a number of metrics compared with uniform mesh refinement.

The algorithm to compute equilibrium solutions to the nonlinear variational systems discussed in Section \ref{model} employs nested iteration (NI) \cite{Starke1}, which begins on a specified coarsest grid. On each NI level, Newton iterations are performed, updating the solution approximation at each step. The stopping criterion for the iterations on each mesh is based on a tolerance of $10^{-4}$ for the approximation's conformance to the first-order optimality conditions in the standard $l_2$ norm. The resulting approximation is then interpolated to a finer grid, where Newton iterations continue. For each iteration, an incomplete Newton correction is performed such that for a given iterate $\vec{u}_k$, the next Newton iterate is given by $\vec{u}_{k+1} = \vec{u}_k + \alpha \delta \vec{u}_h$, where $\alpha \leq 1$. While more sophisticated techniques exist \cite{Emerson3}, this simple approach effectively encourages strict adherence to the unit-length constraint manifold. The damping parameter, $\alpha$, begins at $0.2$ and increases by $0.2$ at each level of NI, to a maximum of $1.0$, as the finer features of the solution become increasingly resolved. For more details on the algorithm, see \cite{Emerson2}. The systems are discretized with bi-quadratic ($Q_2$) elements on quadrilateral meshes for both $\director$ and $\phi$. Finally, the same non-dimensionalization parameters used in \cite{Emerson2} are applied.

On each level, AMR has three stages to produce the next finer mesh: $\text{Estimate} \rightarrow \text{Mark} \rightarrow \text{Refine}$. For each $T \in \mathcal{T}_H$, the local estimator $\Theta_T$ is computed with respect to the coarse approximate solution $\vec{u}_H$. Elements of $\mathcal{T}_H$ are then flagged for refinement through Dorfler marking \cite{Dorfler1}, where $T \in \mathcal{T}_H$ is marked if it is part of a minimal subset $\hat{\mathcal{T}}_H \subset \mathcal{T}_H$ such that $\sum_{T \in \hat{\mathcal{T}}_H} \Theta_T^2 \geq (1- \nu) \sum_{T \in \mathcal{T}_H} \Theta_T^2$. Any marked cells are refined through bisection to produce the next NI mesh. The grid management, discretizations, and adaptive refinement computations are implemented with the widely used \emph{deal.II} finite-element library \cite{BangerthHartmannKanschat2007}.

The simulations here utilize meshes with quadrilateral elements. Therefore, adaptive refinement leads to the existence of hanging nodes. These nodes are dealt with in a standard way by constraining their values with the neighboring regular nodes to maintain continuity along the boundary. Additionally, a $1$-irregular mesh is maintained such that the number of hanging nodes on an edge is at most one. Finally, the theory developed in preceding sections assumes that the studied meshes satisfy the admissibility property. This assumption is valid for the coarsest mesh but, with the introduction of hanging nodes, no longer holds after the first AMR stage. While mesh discretizations employing simplices can maintain admissibility with adaptivity, grids composed purely of quadrilateral elements cannot. Thus, following the first level of refinement, the error estimator is applied heuristically. At present, the \emph{deal.II} library is limited to quadrilateral meshes, but support for simplices is under development. Experimentation with such meshes is planned as important future work.

In order to compare efficiency across different refinement techniques, an approximate work unit (WU) is calculated for each simulation. Assuming the presence of solvers that scale linearly with the number of non-zeros in the matrix, a WU is defined as the sum of the non-zeros in the discretized Hessian for each Newton step over the NI hierarchy divided by the number of non-zeros in a reference fine-grid Hessian. Thus, a WU roughly approximates the work required by any full NI hierarchy in terms of assembling and solving a single linearization step for the reference Hessian when optimally scaling solvers are applied.

Below, the reference Hessian belongs to the finest level of uniform refinement. Say that NI incorporates a hierarchy of $J$ grids with uniform refinement and $\hat{J}$ grids with AMR. Let $n_{z, i}$ and $\hat{n}_{z, i}$ be the number of non-zeros in the discretized Hessian for grid $i$ with uniform or adaptive refinement, respectively. Denote by $m_i$ and $\hat{m}_i$ the number of Newton steps taken on grid $i$ with uniform or adaptive refinement. WUs are computed as
\begin{align*}
WU_{\text{uniform}} = \frac{1}{n_{z, J}} \sum_{i=1}^J m_i \cdot n_{z, i}, & & WU_{\text{AMR}} = \frac{1}{n_{z, J}} \sum_{i=1}^{\hat{J}} \hat{m}_i \cdot \hat{n}_{z, i}.
\end{align*}
While the linear systems here are solved with simple LU decomposition, the reported WUs provide a best-case scaling for comparing the work required between refinement strategies, particularly when an optimal multigrid solver is applied. As a second, rough quantification of efficiency gains, the total wall time to compute a solution is reported. This is denoted by ``Time'' in the results tables. It includes the time overhead for all elements required to produce a final solution, of which the most substantial is constructing and solving the linear systems for each Newton step.

For each of the numerical experiments, the non-dimensionalized physical parameters for $5$CB, a common liquid crystal, are used such that $K_1 = 1$, $K_2 = 0.62903$, $K_3 = 1.32258$, $\epsilon_{\perp} = 7$, and $\epsilon_a = 11.5$. The non-dimensionalized free space permittivity is $\epsilon_0 = 1.42809$, and the flexoelectric constants are $e_s = 1.5$ and $e_b = -1.5$. The penalty parameter is $\zeta = 10^5$. Each of the simulations begins on a $16 \times 16$ mesh followed by $5$ levels of uniform refinement or $6$ levels of AMR. For both experiments, the Dorfler constant is $\nu = 0.1$. Finally, the domain, $\Omega$, is a unit-square.

As noted above, the systems considered in the experiments to follow lack known closed-form solutions from which to compute approximation errors. As a proxy, overkill solutions are computed. That is, for each experiment, an approximate solution is computed for a mesh with $10$ total levels of uniform refinement and more than $16$ million degrees of freedom (DOFs). The overkill solutions are used to produce estimates of approximation error for solutions produced on coarser meshes with both adaptive and uniform refinement. More specifically, if $(\director_o,\phi_o)$ is an overkill solution, $\Vert (\director_o, \phi_o) - (\director_h, \phi_h) \Vert_1$ is reported for approximate solutions $(\director_h, \phi_h)$ as ``H1 Error'' in the results. For comparison, on each mesh, the global H1-error estimate, $\left( \sum_{T \in \triangulation} \Theta_T^2 \right)^{1/2}$, and maximum local H1-error estimate, $\sup_{T \in \triangulation} \Theta_T$, are also reported. It is important to note that the exact values of $C_r$ and $C_e$ from Theorem \ref{final_theorem} are not known, but these estimator values scale favorably with the overkill error in the experiments to follow.

\subsection{A Flexoelectric Field with Patterned Boundaries} \label{flexo_patterned_boundary}

For this problem, no external electric field is applied to the domain. Instead orientational patterning is enforced on the director at the boundaries $y=0.0$ and $y=1.0$. On the boundary, the first component of the director, $n_1$, is uniformly zero. To specify $n_2$ and $n_3$, let $L = -0.95$, $\alpha(x, y) = 4\pi x + \frac{\pi}{2}$, and $\beta(x, y) = 4\pi x - \frac{3\pi}{2}$. Next, define the functions
\begin{align*}
s_1(x, y) &= \frac{L\sin(\alpha(x, y))}{L\cos(\alpha(x, y)) - 1}, & t_1(x, y) &= \frac{L\sin(\alpha(x, y))}{L\cos(\alpha(x, y)) + 1}, \\
s_2(x, y) &= \frac{L\sin(\beta(x, y))}{L\cos(\beta(x, y)) - 1}, & t_2(x, y) &= \frac{L\sin(\beta(x, y))}{L\cos(\beta(x, y)) + 1}.
\end{align*}
Finally, set $\theta_1 = \frac{\pi}{4} + \frac{1}{2}(\arctan(s_1(x, y)) - \arctan(t_1(x, y)))$ and $\theta_2 = \frac{\pi}{4} + \frac{1}{2}(\arctan(s_2(x, y)) - \arctan(t_2(x, y)))$. Then, the boundary conditions are
\begin{align*}
n_2(x, y) = \begin{cases}
    \cos(\theta_1(x, y)) & \text{if } x < 0.5, \\
    \cos(\theta_2(x, y)) & \text{if } x \geq 0.5, \\
\end{cases}
& & 
n_3(x, y) &= \begin{cases}
    \sin(\theta_1(x, y)) & \text{if } x < 0.5, \\
    \sin(\theta_2(x, y)) & \text{if } x \geq 0.5. \\
\end{cases}
\end{align*}
Note that these boundary conditions, and those of the next section, are not exactly representable on the discrete mesh. Thus, at each level of refinement, they are projected onto the mesh boundaries to more closely approximate the true boundary functions with increasing refinement.

The pattern defined by these boundary conditions is seen in Figure \ref{2DPatternRefinement}(\subref{2DPatternRefinement:left1}). The sharp transitions induce elastic deformation on the interior of the domain, subsequently engendering a sustained electric field due to flexoelectric coupling. This field is shown in Figure \ref{2DPatternRefinement}(\subref{2DPatternRefinement:right1}). The points of pattern transition are also the sites with the most elastic and electric activity. As such, it is reasonable to assume that regions near these transitions also house the most difficult physics to capture. In Figure \ref{2DPatternRefinement}(\subref{2DPatternRefinement:center1}), the adaptively refined mesh clearly exhibits a significant focus on these areas.

\begin{figure}[ht!]
\begin{subfigure}{0.33 \textwidth}
	\includegraphics[width=0.97\textwidth, center]{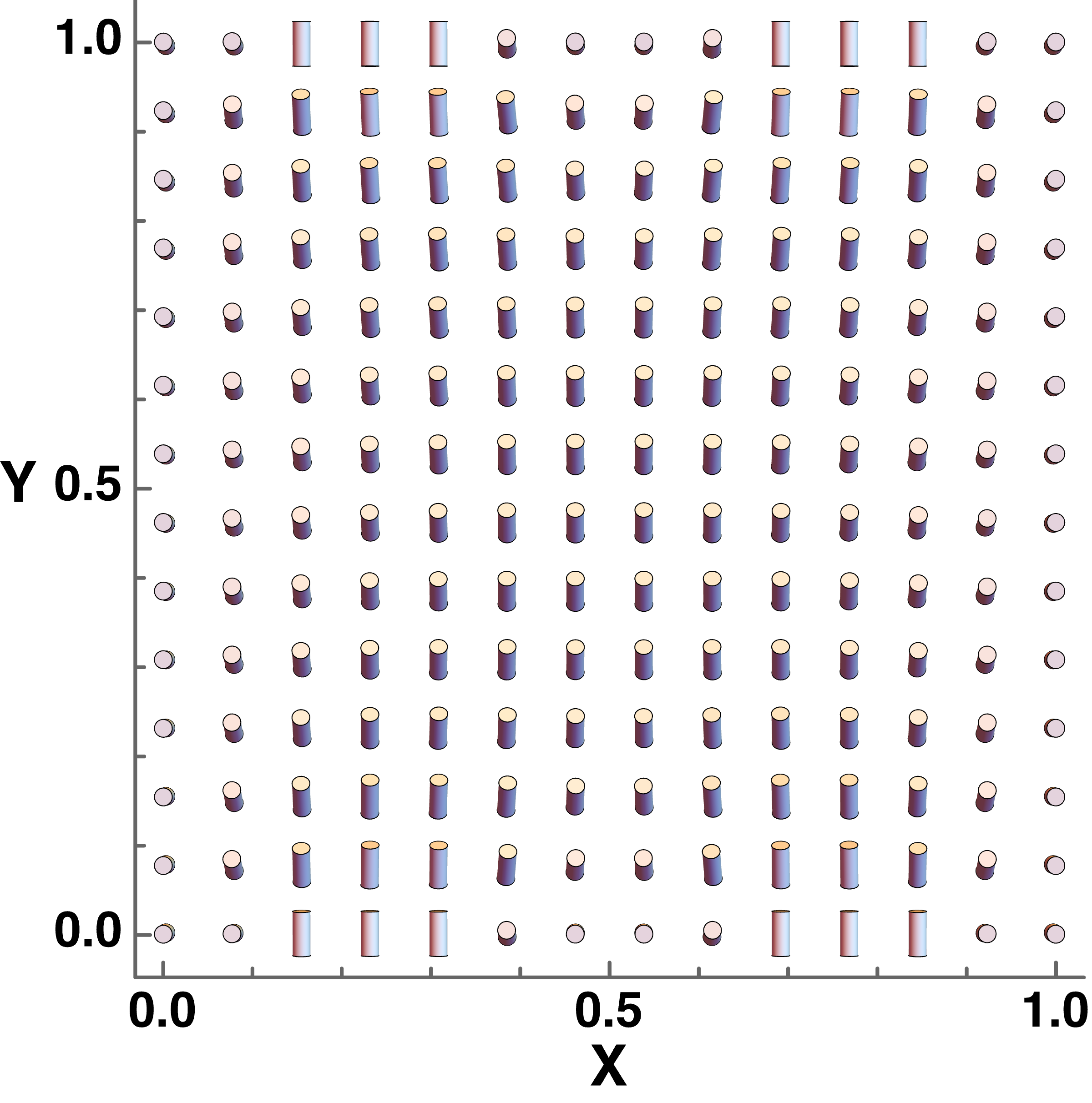} \\ \vspace{-0.25cm}
  \caption{}
  \label{2DPatternRefinement:left1}
\end{subfigure}
\begin{subfigure}{0.33 \textwidth}
	\includegraphics[width=0.99\textwidth, center]{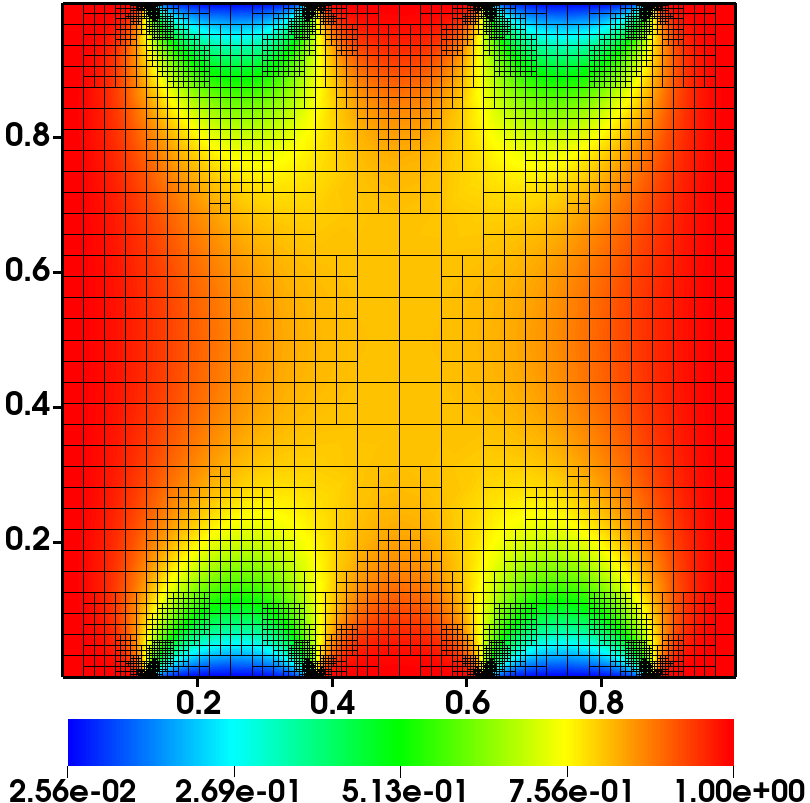}
  \caption{}
  \label{2DPatternRefinement:center1}
\end{subfigure}
\begin{subfigure}{0.33 \textwidth}
	\includegraphics[width=0.99\textwidth, center]{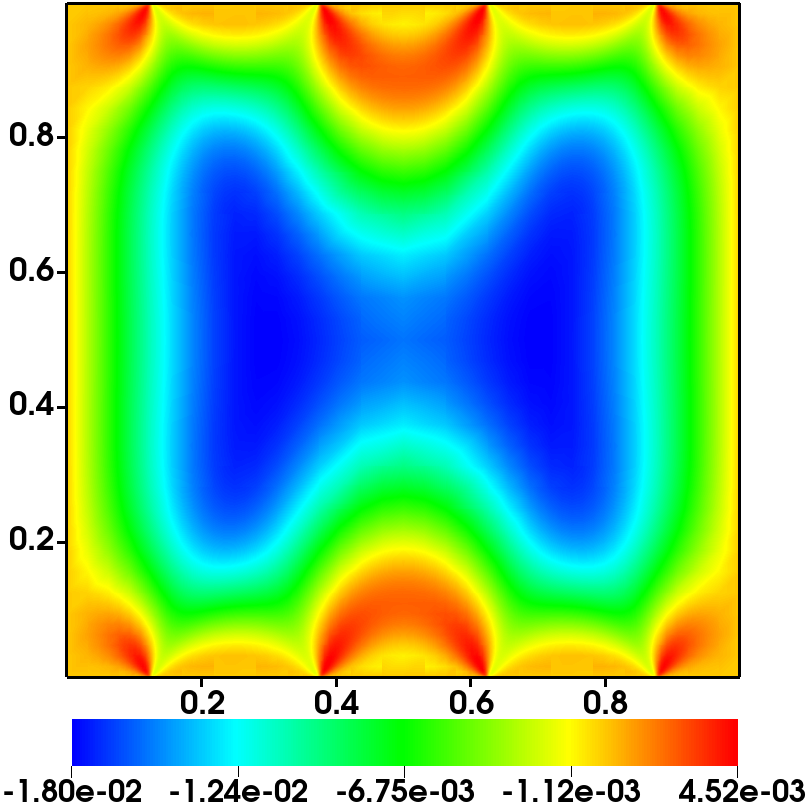}
  \caption{}
  \label{2DPatternRefinement:right1}
\end{subfigure}
\caption{\small{
(\subref{2DPatternRefinement:left1})  Fine-mesh computed solution (restricted for visualization) using Dorfler AMR with $\nu = 0.1$. (\subref{2DPatternRefinement:center1}) Resulting mesh patterns after six levels of adaptive refinement overlaid on the value of $n_3$. (\subref{2DPatternRefinement:right1}) Electric potential produced by the patterned surfaces and flexoelectric coupling.
}}
 \label{2DPatternRefinement}
\end{figure}

Table \ref{PatternElectricAMRStats} presents statistics contrasting the quality of the approximate solution produced with uniform meshes to those computed through AMR. The solution found with AMR demonstrates improved unit-length conformance and equivalent free energy. The column titled ``Gauss'' in Table \ref{PatternElectricAMRStats} reports each solution's local Gauss' law conformance over the domain, measured as $\sum_{T \in \triangulation} \int_T (\diverg \vec{D})^2  \diff{V}$. As no special consideration or care has been taken to strongly enforce Gauss' law conformance outside of adherence to the first-order optimality conditions, the patterned boundary conditions lead to relatively large values. However, conformance for the solution constructed with AMR is markedly better, implying more accurate capture of the relevant physics. These statistics are achieved despite the computation using $40.7$ times fewer DOFs and exhibiting noticeably reduced memory overhead. Finally, AMR computations consume $18$ times fewer WUs and reduce total solve time by $96.9$\%. 

\begin{table}[ht!]
\centering
\caption{\small{Statistics associated with the patterned boundary problem comparing solutions computed with AMR and Dorfler marking to those applying uniform refinement. Numbers correspond to the finest mesh of each approach. The first two columns are the largest director deviations above and below unit-length at the quadrature nodes. The third column is the free energy, as expressed in Equation \eqref{flexofunctional}. ``Gauss'' quantifies a solutions conformance to Gauss' law.
}}
\begin{tabular}{cccccccc}
Strategy & Pos. Dev. & Neg. Dev. & Energy & Gauss & DOFs & WUs & Time\\
\toprule
Uniform & $5.455$e-$03$ & $4.500$e-$03$ & $9.05255$ & $44.416$ & $4,202,500$ & $4.586$ & $1,002$s \\
Adaptive & $5.319$e-$03$ & $4.501$e-$03$ & $9.05263$ & $29.144$ & $103,124$ & $0.246$ & $30.9$s
\end{tabular}
\label{PatternElectricAMRStats}
\end{table} 

Table \ref{PatternElectricNIStats} provides a comparison of the DOFs, Newton steps for convergence, the progression of a solutions free energy, and the H1 error compared to the overkill solution across the NI hierarchy of meshes. The DOFs, and therefore the size of the linear system to solve, grows significantly faster with uniform refinement. However, the number of Newton steps remains constant and the free energy of resolved solutions remains very similar between the uniform and adapted mesh. Moreover, solutions computed with AMR produce similar errors with considerably fewer DOFs. For example, the finest AMR solution exhibits similar error to the finest uniform solution while utilizing $40$ times fewer DOFs. Finally, note that the global and largest element error estimates scale favorably with the error as refinement occurs.

\begin{table}[ht!]
\centering
\caption{\small{A comparison of the DOFs, Newton steps required for convergence, free energy, and H1 error (relative to the overkill solution) across the NI hierarchy for uniform and adaptive mesh refinement, respectively. H1 Est. denotes the total estimator across all elements and $\Theta_T$ represents the maximum estimator value for each mesh.
}}
\resizebox{\textwidth}{!}{
\begin{tabular}{c|llcc|llccll}
\multicolumn{1}{c}{Strategy} & \multicolumn{4}{c}{Uniform} & \multicolumn{6}{c}{Adaptive} \\
\toprule
Mesh \# & DOFs & Steps & Energy & H1 Error & DOFs & Steps & Energy & H1 Error & H1 Est. & $\Theta_T$ \\
\midrule
1 & $4,356$ & $140$ & $7.70541$ & $6.98447$ & $4,356$ & $140$ & $7.70541$ & $6.98447$ & $123.131$ & $30.460$  \\
2 & $16,900$ & $25$ & $8.62811$ & $3.02459$ & $5,384$ & $25$ & $8.57986$ & $3.41378$ & $42.604$ & $28.730$ \\
3 & $66,564$ & $13$ & $9.04516$ & $0.92013$ & $6,472$ & $13$ & $9.02594$ & $1.18045$ & $16.615$ & $7.790$ \\
4 & $264,196$ & $7$ & $9.07285$ & $0.20159$ & $10,388$ & $8$ & $9.07375$ & $0.32266$ & $7.973$ & $3.208$ \\
5 & $1,052,676$ & $3$ & $9.05324$ & $0.05570$ & $18,940$ & $3$ & $9.05480$ & $0.10431$ & $4.033$ & $1.649$ \\
6 & $4,202,500$ & $3$ & $9.05255$ & $0.01396$ & $41,456$ & $3$ & $9.05273$ & $0.03889$ & $2.144$ & $0.648$ \\
7 & -- & -- & -- & -- & $103, 124$ & $2$ & $9.05263$ & $0.01630$ & $1.199$ & $0.192$
\end{tabular}}
\label{PatternElectricNIStats}
\end{table}

\subsection{Electric Field with Sharp Transition}

This experiment considers a system with a large, externally applied, electric field. Uniform boundary conditions are used for the director, fixing $\director = (0, 0, 1)^T$. The electric potential is set to zero along the boundary except for $y = 1.0$ where an approximate square function is used such that $\phi$ rises to roughly $1.5$ on the middle-third of the edge, producing a large electric field with a steep transition near the top boundary. Specifically, reusing $\theta_1(x, y)$ from Section \ref{flexo_patterned_boundary}, the electric potential on the boundary $y=1.0$ is defined as $\phi(x, y) = 1.5\big(\cos(\theta_1(x, y)) - \cos\big(\frac{\pi}{4} -\arctan(L)\big)\big)$.

The effects of the large electric field are seen in Figure \ref{2DSquareRefinement}(\subref{2DSquareRefinement:left1}), which shows the computed solution on the finest mesh with Dorfler AMR. In response to the field, the director deforms to align with the field lines, even near the boundary where elastic resistance is strongest. The electric potential is illustrated in Figure \ref{2DSquareRefinement}(\subref{2DSquareRefinement:right1}). The regions surrounding the rapid transitions in the electric potential contain the most difficult to resolve physics and the largest free-energy contributions. This suggests that a significant portion of the total approximation error will also be present in these areas. In Figure \ref{2DSquareRefinement}(\subref{2DSquareRefinement:center1}), the refinement patterns clearly emphasize the transition regions. As such, finer mesh elements improve resolution therein while areas with dynamics that are likely already well captured by a coarser mesh, such as those near $y=0$, are de-emphasized.

\begin{figure}[ht!]
\begin{subfigure}{0.33 \textwidth}
	\includegraphics[width=0.97\textwidth, center]{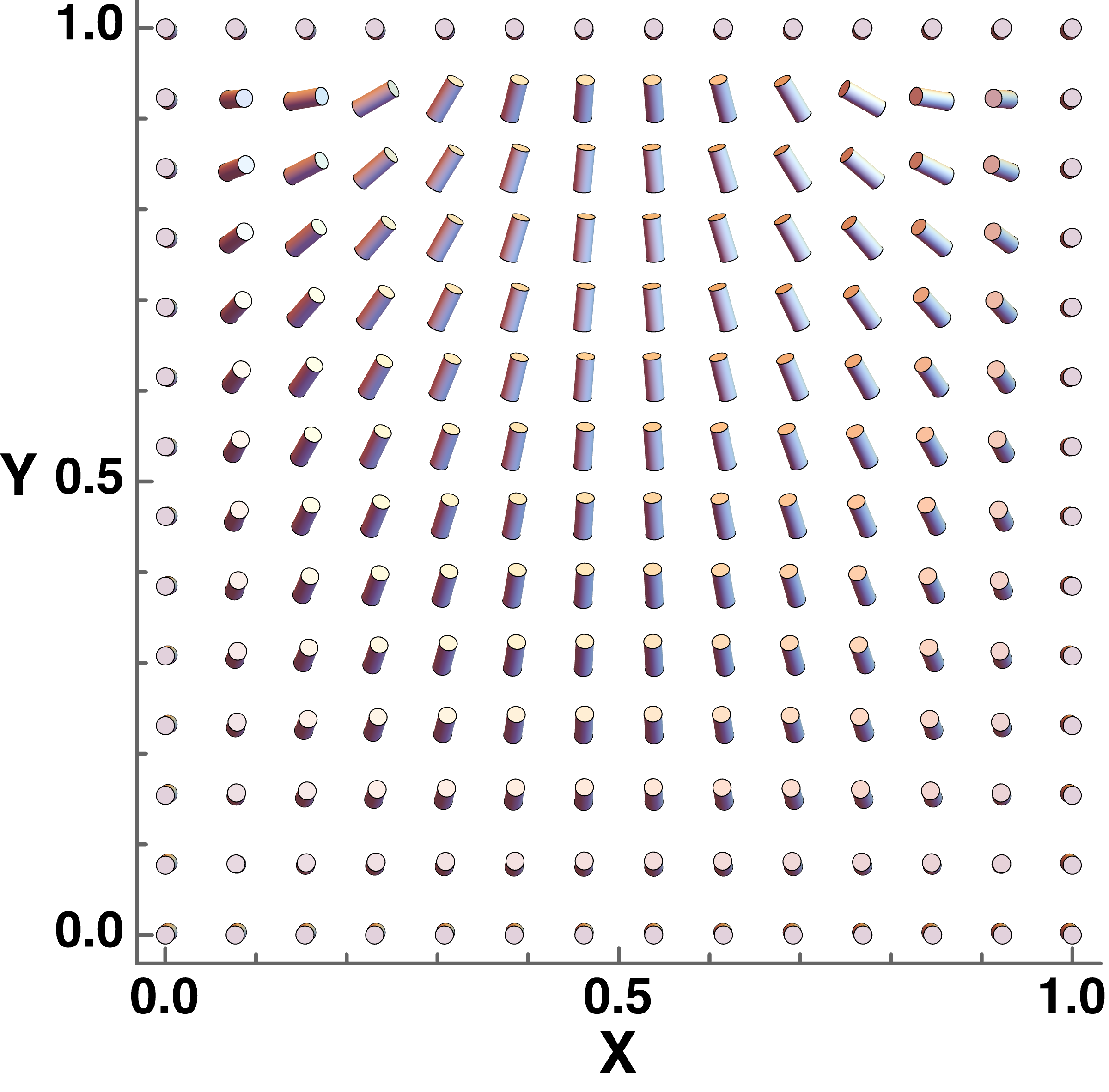} \\ \vspace{-0.25cm}
  \caption{}
  \label{2DSquareRefinement:left1}
\end{subfigure}
\begin{subfigure}{0.33 \textwidth}
	\includegraphics[width=0.99\textwidth, center]{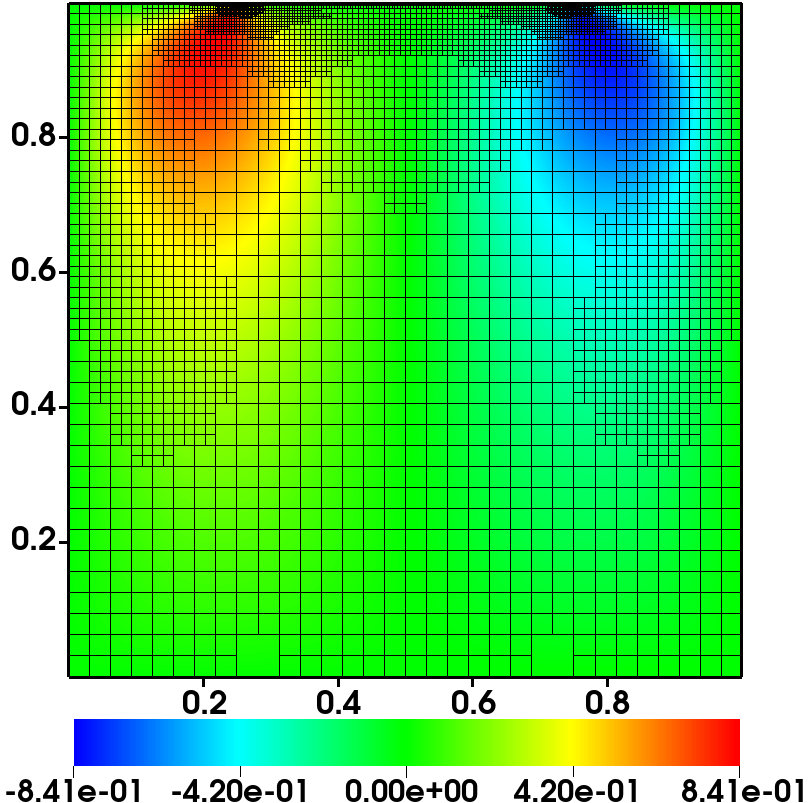}
  \caption{}
  \label{2DSquareRefinement:center1}
\end{subfigure}
\begin{subfigure}{0.33 \textwidth}
	\includegraphics[width=0.99\textwidth, center]{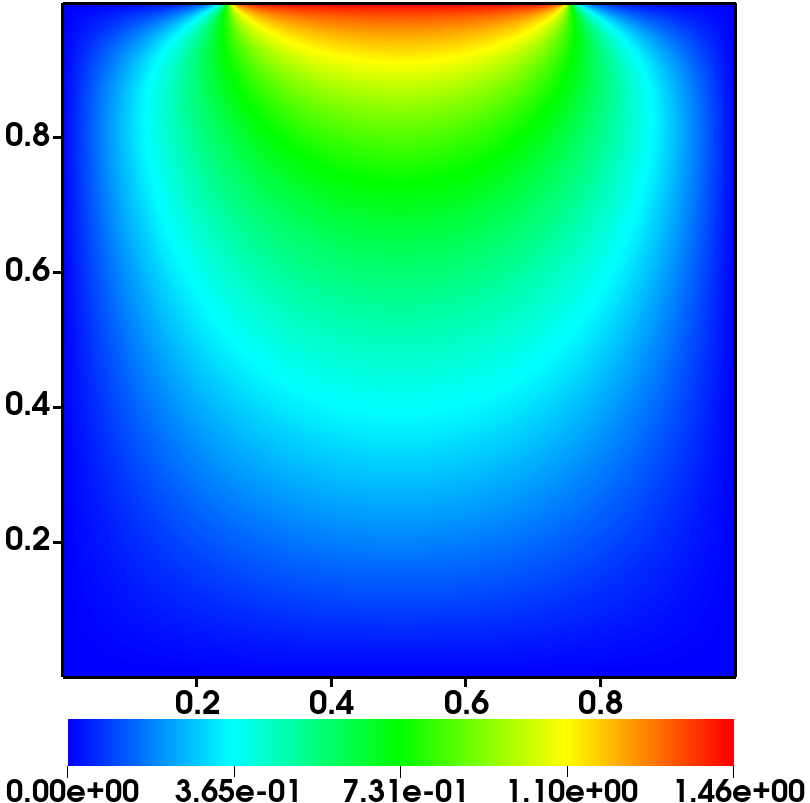}
  \caption{}
  \label{2DSquareRefinement:right1}
\end{subfigure}
\caption{\small{
(\subref{2DSquareRefinement:left1})  Fine-mesh computed solution (restricted for visualization) using Dorfler AMR with $\nu = 0.1$. (\subref{2DSquareRefinement:center1}) Resulting mesh patterns after six levels of adaptive refinement overlaid on the value of $n_1$. (\subref{2DSquareRefinement:right1}) Electric potential produced by the patterned surfaces and flexoelectric coupling.
}}
 \label{2DSquareRefinement}
\end{figure}

In Table \ref{SquareElectricAMRStats}, the AMR solution exhibits slightly tighter unit-length conformance compared to the finest uniform mesh. Furthermore, the solution computed with AMR matches the free energy found with uniform refinement. As in the previous example, the sharp boundary conditions of the electric potential lead to relatively large Gauss' law conformance values. However, conformance of the AMR solution is less than half that of the uniform mesh solution, indicating better physical law conformance. Furthermore, by each measure of computational cost, adaptive refinement yields significant reductions. The solves using AMR consume roughly $12$ times fewer WUs and produce solutions more than an order of magnitude faster. 

\begin{table}[ht!]
\centering
\caption{\small{Statistics associated with the flexoelectric problem comparing solutions computed with AMR and Dorfler marking to those applying uniform refinement. Numbers correspond to each approaches finest mesh. The first two columns are the largest director deviations above and below unit-length at the quadrature nodes. The third column is the free energy, as expressed in Equation \eqref{flexofunctional}. ``Gauss'' quantifies a solutions conformance to Gauss' law.
}}
\begin{tabular}{cccccccc}
 Strategy & Pos. Dev. & Neg. Dev. & Energy & Gauss & DOFs & WUs & Time\\
\toprule
Uniform & $4.447$e-$02$ & $2.700$e-$02$ & $-38.0412$ & $110.417$ & $4,202,500$ & $4.720$ & $1,876$s \\
Adaptive & $4.439$e-$02$ & $2.695$e-$02$ & $-38.0412$ & $37.130$ & $336,628$ & $0.384$ & $57.2$s
\end{tabular}
\label{SquareElectricAMRStats}
\end{table} 

Statistics associated with the hierarchy of NI meshes are reported in Table \ref{SquareElectricNIStats}. While the DOFs scale much more rapidly with uniform refinement, the computed free energies remain comparable. As in the previous experiment, AMR solutions yield similar H1 errors compared to the uniform mesh solutions with significantly less DOFs. For instance, the seventh AMR solutions yields a smaller H1 error while using nearly $50$ times fewer DOFs than the finest uniform mesh. It should be noted that as error becomes more equally distributed, uniform refinement on subsequent meshes emerges as the optimal strategy. However, a maximum proportion of $(1-\nu)$ cells may be refined at a given level with Dorfler marking, preventing fully uniform refinement for a given mesh. As an illustration of this phenomenon, a final step of uniform refinement in the ``Adaptive'' run, delineated as Mesh 8$^*$, is performed. This refinement cuts the error by nearly 80\%, far outpacing the error of the finest uniform mesh while still utilizing an order of magnitude fewer DOFs.

\begin{table}[ht!]
\centering
\caption{\small{A comparison of the DOFs, Newton steps required for convergence, free energy, and H1 error (relative to the overkill solution) across the NI hierarchy for uniform and adaptive mesh refinement, respectively. H1 Est. denotes the total estimator across all elements and $\Theta_T$ represents the maximum estimator value for each mesh.
}}
\resizebox{\textwidth}{!}{
\begin{tabular}{c|llcc|llccll}
\multicolumn{1}{c}{Strategy} & \multicolumn{4}{c}{Uniform} & \multicolumn{6}{c}{Adaptive} \\
\toprule
Mesh & DOFs & Steps & Energy & H1 Error & DOFs & Steps & Energy & H1 Error & H1 Est. & $\Theta_T$ \\
\midrule
1 & $4,356$ & $108$ & $-34.4431$ & $3.41832$ & $4,356$ & $108$ & $-34.4431$ & $3.41832$ & $54.553$ & $25.111$  \\
2 & $16,900$ & $44$ & $-37.3751$ & $1.60086$ & $4,972$ & $45$ & $-37.3177$ & $1.67276$ & $36.379$ & $16.382$ \\
3 & $66,564$ & $15$ & $-38.0264$ & $0.57271$ & $6,052$ & $15$ & $-38.0143$ & $0.62242$ & $19.454$ & $8.019$ \\
4 & $264,196$ & $8$ & $-38.0354$ & $0.15594$ & $8,416$ & $8$ & $-38.0318$ & $0.20374$ & $7.495$ & $2.146$ \\
5 & $1,052,676$ & $3$ & $-38.0402$ & $0.04803$ & $15,356$ & $4$ & $-38.0394$ & $0.07603$ & $2.692$ & $0.366$ \\
6 & $4,202,500$ & $3$ & $-38.0412$ & $0.01347$ & $35,180$ & $3$ & $-38.0410$ & $0.02822$ & $1.052$ & $0.061$ \\
7 & -- & -- & -- & -- & $86,700$ & $2$ & $-38.0412$ & $0.01038$ & $0.168$ & $0.004$ \\
8$^*$ & -- & -- & -- & -- & $336,628$ & $2$ & $-38.0412$ & $0.00235$ & -- & -- \\
\end{tabular}}
\label{SquareElectricNIStats}
\end{table} 

\section{Conclusion and Future Work} \label{conclusions}

In this work, we have discussed an a posteriori error estimator for the electrically and flexoelectrically coupled Frank-Oseen models of nematic liquid crystals with the necessary unit-length constraint enforced via a penalty method. It was shown that the proposed coupled estimator provides a reliable estimate of global approximation error and is an efficient indicator of local error. The estimator is comprised of readily computable, local quantities suitable for use as part of standard cell flagging schemes, including Dorfler marking. Finally, numerical results suggest that it is highly effective in guiding AMR as part of a multilevel nested iteration framework. For each numerical experiment considering flexoelectrically coupled problems with challenging applied electric fields or boundary conditions, AMR with NI produces solutions of better quality and significant reductions in computational cost along a number of metrics. Future work will include extending the theoretical framework to demonstrate reliability and efficiency of the error estimator associated with a Lagrange multiplier formulation. Further, studies incorporating novel multigrid schemes and marking strategies will also be considered. Work investigating whether the present theory can be extended to support irregular meshes, such as those with hanging nodes, as considered in the numerical experiments, is also planned.

\section*{Acknowledgments}

The authors would like to thank Professor Xiaozhe Hu for his helpful suggestions, guidance, and compute resources for the overkill solutions. We would also like to thank Divya Sivasankaran for her careful reading of the manuscript. Finally, thank you to the anonymous reviewers whose suggestions improved this work.

\bibliographystyle{etna}

\bibliography{etna_bib}

\end{document}